\documentclass[fleqn, 11pt, authoryear]{elsarticle}
\biboptions{sort}

\usepackage{hyperref}

\journal{{\color{black}EJOR}}

\usepackage{pifont}

\makeatletter
\def\ps@pprintTitle{%
  \let\@oddhead\@empty
  \let\@evenhead\@empty
  \let\@oddfoot\@empty
  \let\@evenfoot\@oddfoot
}
\makeatother
\usepackage[usenames,dvipsnames]{color}

\usepackage[utf8]{inputenc}
\usepackage[english]{babel}
\usepackage{amsthm}
\usepackage{makecell}
\usepackage{thmtools}
\usepackage{thm-restate}
\usepackage{tabularx}
\usepackage{adjustbox}
\usepackage{hyperref}
\usepackage{xurl}
\usepackage{blindtext}
\usepackage{amsmath,amssymb,amsfonts}
\usepackage{bbm}

\usepackage{cleveref}
\usepackage{fancyhdr}
\usepackage[table]{xcolor}
\usepackage{tikz}
\usetikzlibrary{shapes.geometric, shapes.symbols, positioning}
\usetikzlibrary{backgrounds}
\usetikzlibrary{arrows.meta, fit}
\usetikzlibrary{matrix}
\usepackage{pgfplots}
\usepackage{pgfplotstable}
\usepackage{pgfplots}\usetikzlibrary{patterns}
\usepgfplotslibrary{groupplots}
\usepackage{sankey}
\usepackage{siunitx}

\usepackage{relsize}

\usepackage[short]{optidef}
\usepackage{float}
\graphicspath{{figures/}}
\usepackage{amssymb}
\usepackage{color}
\usepackage{mathtools}
\usepackage{subcaption}
\usepackage{changepage}
\usepackage{geometry}
\usepackage{lipsum}
\usepackage{algorithm}
\usepackage[noend]{algpseudocode}

\algrenewcommand\algorithmicrequire{\textbf{Input:}}
\algrenewcommand\algorithmicensure{\textbf{Output:}}

\usepackage{multirow}
\usepackage{longtable}
\usepackage{supertabular}

\usetikzlibrary{positioning}
\usepgfplotslibrary{dateplot}
\pgfplotsset{compat=1.18}

\usepackage{graphics} 
\usepackage{setspace} 
\usepackage{graphicx} 
\usepackage{lscape}
\usepackage{makeidx} 
\usepackage{pdflscape} 
\usepackage{amsmath,amssymb,amsfonts} 
\usepackage{nccmath}
\usepackage{newtxtext,newtxmath}
\DeclareMathAlphabet{\mathcal}{OMS}{cmsy}{m}{n}
\SetMathAlphabet{\mathcal}{bold}{OMS}{cmsy}{b}{n}
\usepackage{type1cm}
\usepackage{anyfontsize}
\usepackage[nopatch=footnote]{microtype}
\usepackage[final]{pdfpages} 
\usepackage{calc} 
\usepackage{fancyhdr} 
\usepackage{multicol} 
\usepackage[T1]{fontenc}
\usepackage[utf8]{inputenc} 
\usepackage[english]{babel}
\usepackage{threeparttable, supertabular} 
\usepackage{endnotes} 
\usepackage{rotating, rotfloat} 
\usepackage{fullpage} 
\usepackage[usenames,dvipsnames]{color} 
\usepackage{colortbl} 
\usepackage{eurosym}
\usepackage[bottom, flushmargin]{footmisc} 
\usepackage{booktabs} 
\usepackage{tabularx}
\usepackage{array}
\usepackage{mathrsfs}
\usepackage[format=plain, textfont=normal, justification=justified, singlelinecheck=false, font=normal]{caption}
\usepackage{enumitem}
\usepackage{diagbox}
\usepackage{framed}

\allowdisplaybreaks

\newtheorem{proposition}{Proposition}

\newcommand{\doi}[1]{\href{https://doi.org/#1}{\path{https://doi.org/#1}}}

\makeatletter
\renewcommand\section{\@startsection{section}{1}{\z@}%
           {8\p@ \@plus 2\p@ \@minus 2\p@}%
           {4\p@ \@plus 2\p@ \@minus 2\p@}%
           {\normalsize\bfseries\boldmath}}
\renewcommand\subsection{\@startsection{subsection}{2}{\z@}%
           {6\p@ \@plus 2\p@ \@minus 2\p@}%
           {2\p@ \@plus 2\p@ \@minus 2\p@}%
           {\normalfont\normalsize\itshape}}
\renewcommand\subsubsection{\@startsection{subsubsection}{3}{\z@}%
           {6\p@ \@plus 2\p@ \@minus 2\p@}%
           {\p@}%
           {\normalfont\normalsize\itshape}}
\renewcommand\elsparagraph{\@startsection{paragraph}{4}{0\z@}%
           {6\p@ \@plus 2\p@ \@minus 2\p@}%
           {-6\p@}%
           {\normalfont\itshape}}
\g@addto@macro\normalsize{%
  \setlength\abovedisplayskip{6pt plus 2pt minus 2pt}%
  \setlength\belowdisplayskip{6pt plus 2pt minus 2pt}%
  \setlength\abovedisplayshortskip{3pt plus 1pt}%
  \setlength\belowdisplayshortskip{3pt plus 1pt}%
}
\makeatother

\begin{document}

\begin{frontmatter}

\title{Benders Decomposition with Partial Non-Anticipativity Relaxation for Multi-Stage Stochastic Clean Energy Transition Planning}

\author[1]{Ahmet Emir \c{S}ener}

\ead{ahmet.sener@std.bogazici.edu.tr}

\address[1]{Department of Industrial Engineering, Bo\u{g}azi\c{c}i University, Istanbul, Turkey}

\author[2]{Burak Kocuk\corref{cor1}}
\ead{burakkocuk@sabanciuniv.edu}

\author[2]{Tu\u{g}\c{c}e Y\"{u}ksel}
\ead{tugce.yuksel@sabanciuniv.edu}

\address[2]{Faculty of Engineering and Natural Sciences, Sabanc{\i} University, Istanbul, Turkey}
    \cortext[cor1]{Corresponding author. }

\begin{abstract}

\noindent
We study clean energy transition planning for campus-scale integrated electricity-heat systems under both strategic level and operational level uncertainties. We formulate a multi-stage stochastic mixed-integer program that jointly optimizes investment and operational decisions for renewable generation, storage, and heat-transfer technologies whose costs and efficiencies evolve stochastically across stages. To account for short-term operational uncertainty, we further derive a robust reformulation based on box uncertainty sets for demand, renewable generation, and heat-transfer performance. To solve the resulting large-scale model, we develop a Benders decomposition algorithm with partial non-anticipativity relaxation. Investment variables and their non-anticipativity constraints are retained in the master problem while operational variables are assigned to scenario-wise subproblems where their non-anticipativity constraints are relaxed. Non-anticipativity of operational variables is restored only at termination through a smaller linear program. We prove that this correction step can only increase the objective function value by a finite bound. We improve the computational performance of the  algorithm  developed with valid inequalities and a two-phase cut-addition strategy. Using the proposed approach, we solve the Middle East Technical University campus case study at high temporal resolution within a reasonable computational budget, which is not otherwise possible with the extensive form or the classical Benders decomposition. We evaluate the resulting investment plans through rolling-horizon Monte Carlo simulations with an out-of-sample analysis and demonstrate that the added robustness can significantly improve the operational reliability of the transition plans with a moderate increase in total cost. 


\end{abstract}

\begin{keyword}
\texttt{OR in energy, Clean energy transition, Multi-stage stochastic programming, Benders decomposition, Non-anticipativity relaxation}
\end{keyword}

\end{frontmatter}

\newcommand{\smallplus}{\mathrel{\scalebox{0.8}[0.8]{+}}}
\newcommand{\smallminus}{\mathrel{\scalebox{0.9}[0.9]{-}}}
\newcommand{\smalleq}{\mathrel{\scalebox{0.8}[0.8]{=}}}
\newcommand{\smallge}{\mathrel{\scalebox{0.9}[0.9]{$\ge$}}}
\newcommand{\smallle}{\mathrel{\scalebox{0.8}[0.8]{$\le$}}}
\newcommand{\ssetminussobj}{\mathbin{\tikz[baseline={(0,-0.14ex)}] \draw[line width=0.27pt] (0,0.18) -- (0.06,-0.03);}}
\newcommand{\ssetminusseq}{\mathbin{\tikz[baseline={(0,-0.14ex)}] \draw[line width=0.27pt] (0,0.24) -- (0.07,-0.03);}}
\newcommand{\zeroobj}{\tikz[baseline={(0,-0.14ex)}] \node[inner sep=0pt, anchor=base, scale=0.9] {\scriptsize\{0\}};}
\newcommand{\zeroeq}{\tikz[baseline={(0,-0.14ex)}] \node[inner sep=0pt, anchor=base, scale=1.18] {\scriptsize\{0\}};}
\newcommand{\smalltimes}{\mathbin{\mathsmaller{\times}}}
\newcommand{\scriptcompact}[1]{{\scriptsize\spaceskip=.12em\xspaceskip=.2em #1}}

\section{Introduction}

Global sustainability goals and net-zero commitments are intensifying the need to decarbonize large-scale, energy-intensive facilities. Planning a clean energy transition for such facilities is difficult, as it requires transforming capital-intensive infrastructure while maintaining a reliable and continuous energy supply~\citep{davis2018net}. Long-term emission targets make the timing and mix of investments critical, yet these decisions must be made under uncertainty about how technology costs and efficiencies will evolve. These investments must also support energy management at a high temporal resolution, which is itself subject to operational uncertainties such as fluctuating demand and variable technology performance. In this paper, we study the clean energy transition planning problem of a campus-scale integrated electricity and heat system under uncertainty at both the strategic and operational levels.

Campus-scale clean energy transition planning is increasingly framed as an integrated infrastructure design problem that coordinates distributed generation, storage, and energy conversion technologies within interconnected multi-energy systems. Existing studies have examined carbon-neutral transitions, microgrid design, renewable energy and storage integration, and multi-energy system configurations~\citep{Mavromatidis2021,Tian2022,VERGARAZAMBRANO2025101996}, typically combining investment decisions with operational feasibility and infrastructure constraints. Many of these studies, however, rely on deterministic assumptions, which limits their ability to assess how uncertainty affects long-term system performance. Studies that do incorporate stochasticity tend to focus on energy prices, demand growth, renewable resource availability, and policy changes~\citep{Guevara2022,FANG2024123961,AZIMIAN2025126593}. Technology-driven cost reductions and efficiency improvements, by contrast, are treated only implicitly, despite their substantial influence on the timing and portfolio mix of renewable investment decisions. In addition, many of these studies adopt static planning formulations or rely heavily on representative days~\citep{rathi2022,HUYLO2025112736,bacci2025economic} and aggregated temporal resolutions~\citep{Ioannou2019,Guevara2022} to maintain computational tractability. Such simplifications can obscure the chronological dependencies needed to accurately capture storage dynamics, supply--demand balancing, and system feasibility, particularly under high renewable penetration~\citep{gao2023,Marcy2022,CHEN2025117950}.

Solving models that jointly incorporate uncertainty and high-resolution operations is computationally challenging, because long-term investment decisions must remain feasible across many scenario paths. To address this difficulty, the literature offers a range of solution approaches for multi-stage stochastic mixed-integer programs. The extensive-form deterministic equivalent serves as a standard benchmark by reformulating the entire scenario tree into a single deterministic optimization problem~\citep{birge1997introduction}. This approach, however, quickly becomes computationally prohibitive as the number of stages, scenarios, and operational periods grows. To mitigate this burden, sampling-based methods, such as scenario reduction~\citep{dupavcova2003scenario}, sample average approximation~\citep{kleywegt2002sample}, and stochastic decomposition~\citep{higle1991stochastic}, reduce or approximate the scenario representation. Nevertheless, they primarily address the size of the uncertainty model and leave unresolved the structural difficulties arising from integrality, interstage coupling, and non-anticipativity.

A complementary line of research instead exploits the problem structure through decomposition. Scenario decomposition~\citep{mulvey1995new} and progressive hedging~\citep{rockafellar1991scenarios} decompose the problem along scenario paths and coordinate the resulting subproblems by penalizing violations of non-anticipativity, while nested Benders decomposition~\citep{birge1985decomposition} recursively splits the scenario tree into stagewise subproblems and approximates downstream recourse functions through Benders cuts. The application of these methods to problems with integer variables, however, often requires additional mechanisms such as outer branch-and-bound schemes or integer-valid cutting planes. Building on this line of work, stochastic dual dynamic programming~\citep{pereira1991multi} approximates future cost-to-go functions through Benders cuts generated iteratively over forward and backward passes of the scenario tree. Stochastic dual dynamic integer programming~\citep{zou2019stochastic}, together with its recent extensions~\citep{fullner2024new}, adapts this framework to mixed-integer programs by employing integer-valid cuts such as Lagrangian cuts. These methods, however, typically rely on the stagewise independence of the underlying stochastic process, an assumption that does not hold in the setting considered in this paper, where technology costs and efficiencies evolve recursively across stages through advancement multipliers.

Motivated by the modeling gaps and methodological limitations identified above, we formulate a multi-stage stochastic programming (MSSP) model that jointly optimizes investment and operational decisions for renewable generation, storage, and heat-transfer technologies at a high temporal resolution. To capture strategic-level uncertainties in technology costs and efficiencies, we construct a scenario tree that represents the sequential revelation of technological progress over the planning horizon. We further introduce a robust reformulation based on box uncertainty sets to account for short-term operational uncertainties.

To solve the resulting large-scale model, we develop a Benders decomposition algorithm with partial non-anticipativity relaxation, grounded in scenario-based decomposition. The decision variables are split between the master problem and the subproblems according to the difficulty of restoring their associated non-anticipativity constraints. Specifically, investment variables and their non-anticipativity constraints are retained in the master problem, while operational variables are assigned to scenario-wise subproblems in which the corresponding non-anticipativity constraints are temporarily relaxed. Rather than progressively correcting non-anticipativity violations within each Benders iteration, we restore the non-anticipativity of operational decisions at the end of the algorithm by solving a smaller linear program (LP). Owing to the structural properties of the proposed model, this correction step is proven to recover a feasible solution with a bounded increase in the objective function value.

This paper makes three main contributions. (i) We formulate an MSSP model for the clean energy transition of a campus-scale integrated electricity--heat system under uncertain technological progress, and complement it with a robust reformulation that captures operational uncertainty. (ii) We develop a Benders decomposition algorithm with partial non-anticipativity relaxation with a correction step. We further strengthen the algorithm with valid inequalities and a two-phase cut-addition strategy. (iii) We evaluate the nominal and robust investment plans  through rolling-horizon Monte Carlo simulations in an out-of-sample analysis over a broader range of operational realizations. Together, these developments enable us to solve a clean electricity--heat transition case study for the Middle East Technical University (METU) campus in Ankara, Turkey, comprising 9{,}435{,}001 constraints and 10{,}352{,}904 decision variables. In comparison, our earlier work~\citep{csener2025dynamic}, which did not employ the proposed method, could only address the clean electricity transition under uncertain technological advancement without operational uncertainty considerations.

The rest of the paper is organized as follows. Section~\ref{sec:modelling-framework} presents the MSSP model, its robust reformulation, and valid inequalities, while Section~\ref{sec:uncertainty-modeling} characterizes the technological and operational uncertainties they address. Section~\ref{sec:methodology-benders} develops the proposed Benders decomposition algorithm and the out-of-sample evaluation procedure. Section~\ref{sec:computations} reports the computational experiments and the campus-scale case study. Finally, Section~\ref{sec:conclusion-main-findings} concludes the paper.

\section{Optimization Model}
\label{sec:modelling-framework}

\begin{table}[H]
\caption{\small Nomenclature}
\label{tab:nomenclature}
\vspace{-12pt}
\renewcommand{\arraystretch}{1.1}
\begin{framed}
\scriptsize
\noindent
\begin{minipage}[t]{0.48\textwidth}\vspace{0pt}
\noindent\textit{Sets} \\[1ex]
\begin{tabularx}{\linewidth}{@{}l@{\hspace{4pt}}|@{\hspace{4pt}}X@{}}
    $\Psi$ & Set of stages. \\
    $\mathcal{N}$ & Set of nodes in the scenario tree, $\mathcal{N} = \{0, \ldots, N\}$. \\
    $\mathcal{N}_{l}$ & Set of leaf nodes in the scenario tree. \\
    $\mathcal{T}$ & Set of investment periods, $\mathcal{T} = \{0, \ldots, T\}$. \\
    $\mathcal{T}_n$ & Subset of investment periods associated with node $n \in \mathcal{N}$. \\
    $\mathcal{Q}$ & Set of operational sub-periods within each investment period, $\mathcal{Q} = \{1, \ldots, Q\}$. \\
    $\mathcal{R}$ & Set of energy categories, $\mathcal{R} = \{\text{electricity},\ \text{heat}\}$. \\
    $\mathcal{J}$ & Set of energy generation technologies. \\
    $\mathcal{S}$ & Set of energy storage technologies. \\
    $\mathcal{A}$ & Set of heat transfer technologies. \\
    $\mathcal{U}$ & Set of all technologies, $\mathcal{U} = \mathcal{J} \cup \mathcal{S} \cup \mathcal{A}$. \\
    $\mathcal{T}_{u,[t)}$ & Set of periods $t'$ during which technology $u \in \mathcal{U}$, installed in period $t \in \mathcal{T}$, remains operational, i.e., $t \leq t' < t + \tau_{(u,t)}$. \\
    $\mathcal{T}'_{u,[t')}$ & Set of possible installation periods $t$ of technology $u \in \mathcal{U}$ that is operational in period $t' \in \mathcal{T}$, i.e., $t \leq t' < t + \tau_{(u,t)}$. \\
\end{tabularx}

\vspace{0.1cm}
\noindent\textit{Decision variables} \\[1ex]
\begin{tabularx}{\linewidth}{@{}l@{\hspace{4pt}}|@{\hspace{4pt}}X@{}}
    $x_{u,t,n}$ & Number of installations of technology  $u \in \mathcal{U}$ in period $t \in \mathcal{T}_n$ at node $n \in \mathcal{N}$. \\
    $h_{r,q,t,n}$ & Energy $r \in \mathcal{R}$ stored at the end of sub-period $q \in \mathcal{Q}$ in period $t \in \mathcal{T}_n$ at node $n \in \mathcal{N}$, carried forward to the next sub-period. \\
    $h_{r,q,t,n}^{\smallplus}$ & Energy $r \in \mathcal{R}$ charged into storage during sub-period $q \in \mathcal{Q}$ in period $t \in \mathcal{T}_n$ at node $n \in \mathcal{N}$. \\
    $h_{r,q,t,n}^{\smallminus}$ & Energy $r \in \mathcal{R}$ discharged from storage during sub-period $q \in \mathcal{Q}$ in period $t \in \mathcal{T}_n$ at node $n \in \mathcal{N}$. \\
    $p_{r,q,t,n}$ & Energy $r \in \mathcal{R}$ purchased exogenously in sub-period $q \in \mathcal{Q}$ in period $t \in \mathcal{T}_n$ at node $n \in \mathcal{N}$. \\
    $o_{r,q,t,n}$ & Energy $r \in \mathcal{R}$ used from inventory and generation in sub-period $q \in \mathcal{Q}$ of period $t \in \mathcal{T}_n$ at node $n \in \mathcal{N}$. \\
    $m_{a,q,t,t',n}$ & Heat transferred in sub-period $q \in \mathcal{Q}$ in period $t \in \mathcal{T}_n$ at node $n \in \mathcal{N}$ by technology $a \in \mathcal{A}$ that was installed in period $t \in \mathcal{T}'_{a,[t')}$. \\
\end{tabularx}

\vspace{0.1cm}
\noindent\textit{Parameters} \\[1ex]
\begin{tabularx}{\linewidth}{@{}l@{\hspace{4pt}}|@{\hspace{4pt}}X@{}}
    $\Gamma_{j,r,q,t,t',n}$ & Energy $r \in \mathcal{R}$ generation of technology $j \in \mathcal{J}$, in sub-period $q \in \mathcal{Q}$ of period $t' \in \mathcal{T}_n$ at node $n \in \mathcal{N}$, for installations made in period $t \in \mathcal{T}'_{j,[t')}$. \\
    $\beta_{\textnormal{nominal}}$ & Nominal discount rate applied to the periods. \\
    $\zeta$ & Inflation rate applied to the periods.
\end{tabularx}

\end{minipage}
\hfill\vrule\hfill
\begin{minipage}[t]{0.48\textwidth}\vspace{0pt}
\begin{tabularx}{\linewidth}{@{}l@{\hspace{4pt}}|@{\hspace{4pt}}X@{}}
    $\beta_{\textnormal{real}}$ & Real discount rate, calculated as $\beta_{\textnormal{real}} = \frac{1 + \beta_{\textnormal{nominal}}}{1 + \zeta} - 1$. \\
    $\beta$ & Discount factor, calculated as $\beta = \frac{1}{1 + \beta_{\textnormal{real}}}$. \\
    $\pi_n$ & Probability of node $n \in \mathcal{N}$ in the scenario tree. \\
    $\psi_n$ & Stage that node $n \in \mathcal{N}$ belongs to. \\
    $\mu_{n,t}$ & Ancestor node of $n \in \mathcal{N}$ in period $t \in \mathcal{T}$. \\
    $\kappa_{q,t,n}$ & Sub-period, period, and node immediately preceding sub-period $q \in \mathcal{Q}$ in period $t \in \mathcal{T}_n$ at node $n \in \mathcal{N}$. \\
    $\hat{\kappa}_{q,t,n}$ & Period immediately preceding sub-period $q \in \mathcal{Q}$ in period $t \in \mathcal{T}_n$ at node $n \in \mathcal{N}$. \\
    $\alpha_{u,n}$ & Installation cost of technology $u \in \mathcal{U}$ at node $n \in \mathcal{N}$. \\
    $\hat{\alpha}_{u,t}$ & Total O\&M cost over the remaining planning horizon for technology $u \in \mathcal{U}$ installed in period $t \in \mathcal{T}$. \\
    $\lambda_{r,t}$ & Cost of purchasing energy $r \in \mathcal{R}$ exogenously in period $t \in \mathcal{T}$. \\
    $\delta_{r,q,t}$ & Energy $r \in \mathcal{R}$ demand in sub-period $q \in \mathcal{Q}$ of period $t \in \mathcal{T}$. \\
    $\upsilon_{s,r,t,t'}$ & Energy $r \in \mathcal{R}$ storage capacity of technology $s \in \mathcal{S}$, during period $t' \in \mathcal{T}$ for installations made in period $t \in \mathcal{T}'_{s,[t')}$. \\
    $\varphi_{a}$ & Heat transfer capacity of technology $a \in \mathcal{A}$.\\
    $\sigma_{u,t,t'}$ & Degradation rate of technology $u \in \mathcal{U}$, in period $t' \in \mathcal{T}$, for units installed in period $t \in \mathcal{T}'_{u,[t')}$. \\
    $\eta^{\smallplus}_{r}, \eta^{\smallminus}_{r}$ & Charging and discharging efficiencies of energy $r \in \mathcal{R}$ storage technologies, respectively. \\
    $\eta_{r}$ & Self discharge rate of energy $r \in \mathcal{R}$ storage technologies. \\
    $\mathrm{cop}_{a,q,t,n}$ & Coefficient of performance in sub-period $q \in \mathcal{Q}$ for technology $a \in \mathcal{A}$ installed in period $t \in \mathcal{T}_n$ at node $n \in \mathcal{N}$. \\
    $\rho_{u,t}$ & Spatial requirement of technology $u \in \mathcal{U}$, in period $t \in \mathcal{T}$. \\
    $\tau_{u,t}$ & Economic lifetime of technology $u \in \mathcal{U}$ installed in period $t \in \mathcal{T}$. \\
    $\iota_{u}$ & Number of technology $u \in \mathcal{U}$ existing at the beginning of the planning horizon. \\
    $\gamma_{r}$ & Maximum fraction of annual demand for $r \in \mathcal{R}$ that can be met by emission-causing exogenous purchases in the final period. \\
    $\phi_t$ & Available budget for installations in period $t \in \mathcal{T}$. \\
    $\varkappa_{t}$ & Cumulative maximum installation area available up to period $t \in \mathcal{T}$.
\end{tabularx}
\end{minipage}
\end{framed}
\end{table}

\subsection{Multi-Stage Stochastic Programming Model}
\label{subsec:mssp-model}

The deterministic equivalent of the proposed MSSP model is formulated using the notation defined in Table~\ref{tab:nomenclature}. Technology costs and efficiency-related performance profiles, including renewable energy generation and coefficients of performance, are indexed by the nodes in the scenario-tree $\mathcal{N}$ to capture their stage-wise realization under uncertainty. Each node $n \in \mathcal{N}$ comprises a set of periods $\mathcal{T}_n$, in which long-term technology investment decisions are made, and each period is further divided into sub-periods $\mathcal{Q}$, in which short-term operational decisions are made to meet energy demand. The resulting mixed-integer linear program given below jointly optimizes these investment and operational decisions for clean energy transition planning in an integrated electricity--heat system.

{\small
\begin{subequations}\label{eq:stochasticModel}
\begin{flalign}
    & \text{minimize} \ \sum_{n \in \mathcal{N}\ssetminussobj \zeroobj} \sum_{t \in \mathcal{T}_n} \pi_n \beta^{(t-1)} \left( \sum_{u \in \mathcal{U}} (\alpha_{u,n} \smallplus \hat{\alpha}_{u,t}) \ x_{u,t,n} \smallplus \sum_{r \in \mathcal{R}} \sum_{q \in \mathcal{Q}} \lambda_{r,t} \ p_{r,q,t,n} \right) \label{eq:objFunc} 
\end{flalign}
\vspace{-0.038\textwidth}
\begin{flalign}
    & \text{subject to } \notag \\
    & o_{r,q,t',n} \smallplus  \sum_{a \in \mathcal{A}} \sum_{t \in \mathcal{T}'_{a,[t')}} \mathbbm{1}_{a,r,q,t,t',n} \ m_{a,q,t,t',n} \smallge \mathit{\delta}_{r,q,t'}
    & n \in \mathcal{N} \ssetminusseq \zeroeq, t' \in \mathcal{T}_n, q \in \mathcal{Q}, r \in \mathcal{R} \label{eq:demandConstr} \\
    & p_{r,q,t',n} \smallminus h_{r,q,t',n}^{\smallplus} \smallplus h_{r,q,t',n}^{\smallminus} \smallplus \sum_{j \in \mathcal{J}} \sum_{t \in \mathcal{T}'_{j,[t')}} \Gamma_{j,r,q,t,t',n} \ x_{j,t,\mu_{t,n}} \smallge o_{r,q,t',n}
    & n \in \mathcal{N} \ssetminusseq \zeroeq, t' \in \mathcal{T}_n, q \in \mathcal{Q}, r \in \mathcal{R} \label{eq:demandConstr2}
    \end{flalign}
    \vspace{-0.05\textwidth}
    \begin{flalign}
    & h_{r,q,t',n} = \eta_{r} h_{r,\kappa_{q,t',n}} \smallplus \eta^{\smallplus}_{r} h_{r,q,t',n}^{\smallplus} \smallminus \frac{h_{r,q,t',n}^{\smallminus}}{\eta^{\smallminus}_{r}}
    & n \in \mathcal{N} \ssetminusseq \zeroeq, t' \in \mathcal{T}_n, q \in \mathcal{Q}, r \in \mathcal{R} \label{eq:storagebalanceConstr} \\
    & h_{r,q,t',n} \leq \sum_{s \in \mathcal{S}} \sum_{t \in \mathcal{T}'_{s,[t')}} \upsilon_{s,r,t,t'} \ x_{s,t,\mu_{t,n}}
    & n \in \mathcal{N}, t' \in \mathcal{T}_n, q \in \mathcal{Q}, r \in \mathcal{R} \label{eq:storagecapacityConstr} \\
    & m_{a,q,t,t',n} \leq \varphi_{a} \ x_{a,t,\mu_{t,n}}
    & \quad \ \ n \in \mathcal{N}, t' \in \mathcal{T}_n, a \in \mathcal{A}, t \in \mathcal{T}'_{a,[t')}, q \in \mathcal{Q} \label{eq:heattransferConstr} \\
    & \sum_{q \in \mathcal{Q}} p_{r,q,T,n} \leq \gamma_{r} \sum_{q \in \mathcal{Q}} \mathit{\delta}_{r,q,T}
    & n \in \mathcal{N}_{l}, r \in \mathcal{R} \label{eq:emissionConstr} \\
    & \sum_{u \in \mathcal{U}} \alpha_{u,n} \ x_{u,t,n} \leq \phi_{t}
    & n \in \mathcal{N} \ssetminusseq \zeroeq, t \in \mathcal{T}_n \label{eq:budgetConstr} \\
    & \sum_{u \in \mathcal{U}} \sum_{t \in \mathcal{T}'_{u,[t')}} \rho_{u,t} \ x_{u,t,\mu_{t,n}} \leq \varkappa_{t'}
    & n \in \mathcal{N} \ssetminusseq \zeroeq, t' \in \mathcal{T}_n \label{eq:spatialConstr} \\[-0.45em]
    & x_{u,0,0} = \iota_{u}
    & u \in \mathcal{U} \label{eq:initializationConstr} \\
    & \left\lfloor \frac{\phi_t}{\alpha_{u,n}} \right\rfloor \smallge x_{u,t,n}, \quad x_{u,t,n} \in \mathbb{Z}_+
    & n \in \mathcal{N}, t \in \mathcal{T}_n, u \in \mathcal{U} \label{eq:VarDomain1}\\
    & m_{a,q,t,t',n} \in \mathbb{R}_+ 
    & n \in \mathcal{N}, t' \in \mathcal{T}_n, t \in \mathcal{T}'_{a,[t')}, q \in \mathcal{Q}, a \in \mathcal{A} \label{eq:VarDomain2} \\
    & p_{r,q,t,n}, \ h_{r,q,t,n}, \ h_{r,q,t,n}^{\smallplus}, \ h_{r,q,t,n}^{\smallminus}, \ o_{r,q,t,n} \in \mathbb{R}_+ 
    & n \in \mathcal{N}, t \in \mathcal{T}_n, q \in \mathcal{Q}, r \in \mathcal{R}.\label{eq:VarDomain3}
\end{flalign}
\end{subequations}
}

The objective function~\eqref{eq:objFunc} minimizes the expected discounted total cost over the planning horizon. The cost components include technology installation costs, operation and maintenance (O\&M) costs, and exogenous energy procurement costs.
In each sub-period, energy demand must be satisfied through a combination of on-site generation, exogenous energy procurement, storage discharging, and heat transfer operations. We assume that energy supplied by heat transfer technologies cannot be charged into heat storage. To impose this assumption, we introduce auxiliary decision variables $o_{r,q,t,n} \smallge 0$ in the demand satisfaction constraint sets~\eqref{eq:demandConstr} and~\eqref{eq:demandConstr2}, whose non-negativity restricts storage charging to be supplied only by on-site generation and exogenous procurement. The indicator variables defined in~\eqref{eq:indicatorheattransfer} assign energy-category-specific coefficients to heat transfer operations, taking distinct values for heat and electricity, respectively.

{
\small
\begin{flalign}
\label{eq:indicatorheattransfer}
& \mathbbm{1}_{a,r,q,t,t',n} =
\begin{cases}
1-\sigma_{a,t,t'}, \quad \quad & \text{if } r=\mathrm{heat} \\[3pt]
\dfrac{-1}{\mathrm{cop}_{a,q,t,n}}, \ \ \quad \quad & \text{if } r=\mathrm{electricity}
\end{cases} & a \in \mathcal{A}, q \in \mathcal{Q}, n \in \mathcal{N}, t' \in \mathcal{T}_n, t \in \mathcal{T}'_{a,[t')}
\end{flalign}
}

Storage balance constraints~\eqref{eq:storagebalanceConstr} govern storage dynamics by accounting for charging and discharging efficiencies as well as self-discharge losses. The stored energy levels and heat transfer capacities are bounded by the installed technology capacities through constraints~\eqref{eq:storagecapacityConstr} and~\eqref{eq:heattransferConstr}, respectively. Notably, heat transfer decision variables and capacities are modeled in a disaggregated form to accurately capture technology degradation and the coefficient of performance.

To ensure environmental compliance, the percentage of annual demand that can be met by emission-causing exogenous energy purchases is constrained by~\eqref{eq:emissionConstr} for each energy category. Financial and spatial feasibility are enforced through budget constraints~\eqref{eq:budgetConstr}, which cap annual installations based on available capital, and spatial constraints~\eqref{eq:spatialConstr}, which limit the cumulative footprint of installed systems. The integration of pre-existing technologies at the start of the planning horizon is handled by initialization constraints~\eqref{eq:initializationConstr}.

Installation decisions for energy technologies are represented by the non-negative integer variables defined in~\eqref{eq:VarDomain1}. Their upper bounds, which follow from the available budget and the corresponding per-unit installation costs, are imposed explicitly as valid bounds on the integer variables. The operational decision variables are defined as non-negative continuous variables in~\eqref{eq:VarDomain2}--\eqref{eq:VarDomain3}.

\subsection{Robust Reformulation}
\label{subsec:robust-reformulation}

Alongside the investment plan obtained by solving model~\eqref{eq:stochasticModel} with nominal operational profiles, we construct more conservative plans through a robust reformulation that protects the investment decisions against unfavorable realizations of the uncertain operational parameters: energy demand, renewable generation, and the coefficient of performance. 
We enclose these parameters in a box uncertainty set  $[\,v_{d,i} - \epsilon\,\hat{v}_{d,i},\ v_{d,i} + \epsilon\,\hat{v}_{d,i}\,]$, where the nominal value and its deviation are respectively denoted as   $v_{d,i}$ and $\hat{v}_{d,i} \smallge 0$, and the robustness parameter is   $\epsilon \smallge 0$. 
The uncertain parameters appear only in the demand satisfaction constraints~\eqref{eq:demandConstr} and~\eqref{eq:demandConstr2}, where each enters monotonically: feasibility is most tightly constrained when demand is high and when renewable generation and the coefficient of performance are low. The worst case over the uncertainty set is therefore attained at a vertex of the box, so it suffices to protect against this single worst-case realization rather than the entire set. Accordingly, we displace each parameter to the adverse end of its interval---energy demand to $v_{d,i} + \epsilon\,\hat{v}_{d,i}$, and renewable generation and the coefficient of performance to $v_{d,i} - \epsilon\,\hat{v}_{d,i}$---and obtain the investment decisions under these displaced values. Setting $\epsilon = 0$ collapses the uncertainty set to the nominal point and recovers the original model, whereas larger values enlarge the set and yield increasingly conservative investment decisions. The robust model therefore retains the structure of the nominal one and can be solved in the same way.

\subsection{Derivation of Valid Inequalities}

In this subsection, we derive a family of valid inequalities that are expressed solely in terms of the investment decision variables.

\begin{proposition}
For each $r \in \mathcal{R}$, the following  inequalities are valid for the MSSP model~\eqref{eq:stochasticModel} when $\gamma_{r} = 0$:

{\footnotesize
\begin{flalign}
& \eta^{\smallminus}_{r} \eta_{r} \sum_{s \in \mathcal{S}} \sum_{t \in \mathcal{T}'_{s,[\hat{\kappa}_{\underline{q},T,n})}} \upsilon_{s,r,t,\hat{\kappa}_{\underline{q},T,n}} x_{s,t,\mu_{t,n}} \smallplus \sum_{q=\underline{q}}^{\overline{q}} \sum_{j \in \mathcal{J}} \sum_{t \in \mathcal{T}'_{j,[T)}} \Gamma_{j,r,q,t,T,n} x_{j,t,\mu_{t,n}} & \notag \\
& \smallplus \sum_{q=\underline{q}}^{\overline{q}}  \sum_{a \in \mathcal{A}} \sum_{t \in \mathcal{T}'_{a,[T)}} \varsigma_{a,r,t,T} x_{a,t,\mu_{t,n}} \smallge \sum_{q=\underline{q}}^{\overline{q}} \mathit{\delta}_{r,q,T}
& r \in \mathcal{R}, n \in \mathcal{N}_{l}, \underline{q} \in \mathcal{Q}, \overline{q} \in  \{\underline{q}, \ldots, Q\} \label{eq:validineq10} \end{flalign}
}

\noindent
where the coefficients $\varsigma_{a,r,t,T}$ are defined as follows:

{\footnotesize
\begin{flalign}
& \varsigma_{a,r,t,T} =
\begin{cases}
\varphi_{a} \ (1\smallminus\sigma_{a,t,T}), \quad \quad &\text{if } r=\mathrm{heat} \\[3pt]
0, \quad \quad \quad \quad \quad \quad \quad &\text{if } r=\mathrm{electricity}
\end{cases} & a \in \mathcal{A}, t \in \mathcal{T}'_{a,[T)} \label{eq:varsigmaint}
\end{flalign}
}

\end{proposition}

\begin{proof}

Consider a leaf node $n \in \mathcal{N}_{l}$ and an energy category $r \in \mathcal{R}$. For $\gamma_{r} = 0$, the emission constraints~\eqref{eq:emissionConstr} prohibit exogenous energy purchases in period~$T$. Consequently, the variables $p_{r,q,T,n}$ can be eliminated from~\eqref{eq:demandConstr2}. Substituting the discharge variables $h_{r,q,T,n}^{\smallminus}$ using the storage balance constraints~\eqref{eq:storagebalanceConstr} yields~\eqref{eq:validineq1}:

{\footnotesize
\begin{flalign}
    & \eta^{\smallminus}_{r} \eta_{r} h_{r,\kappa_{q,T,n}} \smallminus \eta^{\smallminus}_{r} h_{r,q,T,n} \smallplus (\eta^{\smallminus}_{r} \eta^{\smallplus}_{r} \smallminus 1) h_{r,q,T,n}^{\smallplus} \smallplus \sum_{j \in \mathcal{J}} \sum_{t \in \mathcal{T}'_{j,[T)}} \Gamma_{j,r,q,t,T,n} \ x_{j,t,\mu_{t,n}} \smallge o_{r,q,T,n}
    & q \in \mathcal{Q}\label{eq:validineq1}
\end{flalign}
}

Given that $\eta_r^{\smallplus},\eta_r^{\smallminus}\in(0,1)$, the coefficient $(\eta_r^{\smallminus}\eta_r^{\smallplus}\smallminus1)$ of $h_{r,q,T,n}^{\smallplus}$ in~\eqref{eq:validineq1} is strictly negative. Since $h_{r,q,T,n}^{\smallplus}\ge 0$ and the inequality is of the form ``$\smallge$'', dropping this term preserves its validity, resulting in~\eqref{eq:validineq2}:

{\footnotesize
\begin{flalign}
    & \eta^{\smallminus}_{r} (\eta_{r} h_{r,\kappa_{q,T,n}} \smallminus h_{r,q,T,n}) \smallplus \sum_{j \in \mathcal{J}} \sum_{t \in \mathcal{T}'_{j,[T)}} \Gamma_{j,r,q,t,T,n} \ x_{j,t,\mu_{t,n}} \smallge o_{r,q,T,n}
    & q \in \mathcal{Q}\label{eq:validineq2}
\end{flalign}
}

Summing~\eqref{eq:validineq2} over the interval $q \in \{\underline{q},\ldots,\overline{q}\}$ for any $\underline{q} \in \mathcal{Q}$ and $\overline{q} \in \{\underline{q},\ldots,Q\}$ yields~\eqref{eq:validineq3}:

{\footnotesize
\begin{flalign}
    & \sum_{q=\underline{q}}^{\overline{q}} \left( \eta^{\smallminus}_{r} (\eta_{r} h_{r,\kappa_{q,T,n}} \smallminus h_{r,q,T,n}) \smallplus \sum_{j \in \mathcal{J}} \sum_{t \in \mathcal{T}'_{j,[T)}} \Gamma_{j,r,q,t,T,n} \ x_{j,t,\mu_{t,n}} \right) \smallge \sum_{q=\underline{q}}^{\overline{q}} o_{r,q,T,n}
    & \underline{q} \in \mathcal{Q}, \overline{q} \in  \{\underline{q}, \ldots, Q\}\label{eq:validineq3}
\end{flalign}
}

Upon rearranging the terms, the interior storage variables $(h_{r,q,T,n})_{q=\underline{q}+1}^{\overline{q}-1}$ take the coefficient $\eta_r^{\smallminus}(\eta_r\smallminus1)$, and $h_{r,\overline{q},T,n}$ takes the coefficient $\smallminus\eta_r^{\smallminus}$. As these coefficients are negative and the storage variables are nonnegative, dropping these terms maintains the validity of the inequality, providing~\eqref{eq:validineq5}:

{\footnotesize
\begin{flalign}
    & \eta^{\smallminus}_{r} \eta_{r} h_{r,\kappa_{\underline{q},T,n}} \smallplus \sum_{q=\underline{q}}^{\overline{q}} \sum_{j \in \mathcal{J}} \sum_{t \in \mathcal{T}'_{j,[T)}} \Gamma_{j,r,q,t,T,n} \ x_{j,t,\mu_{t,n}} \smallge \sum_{q=\underline{q}}^{\overline{q}} o_{r,q,T,n}
    & \underline{q} \in \mathcal{Q}, \overline{q} \in  \{\underline{q}, \ldots, Q\}\label{eq:validineq5}
\end{flalign}
}

Next, multiplying the storage capacity constraints~\eqref{eq:storagecapacityConstr} by $\eta_r^{\smallminus}\eta_r$ gives~\eqref{eq:validineq6}.

{\footnotesize
\begin{flalign}
    & \eta^{\smallminus}_{r} \eta_{r} \sum_{s \in \mathcal{S}} \sum_{t \in \mathcal{T}'_{s,[\hat{\kappa}_{\underline{q},T,n})}} \upsilon_{s,r,t,\hat{\kappa}_{\underline{q},T,n}} \ x_{s,t,\mu_{t,n}} \ge \eta^{\smallminus}_{r} \eta_{r} h_{r,\kappa_{\underline{q},T,n}}
    & \underline{q} \in \mathcal{Q}\label{eq:validineq6}
\end{flalign}
}

Combining~\eqref{eq:validineq5} and~\eqref{eq:validineq6} eliminates the remaining storage variables, yielding~\eqref{eq:validineq7}:

{\footnotesize
\begin{flalign}
    & \eta^{\smallminus}_{r} \eta_{r} \sum_{s \in \mathcal{S}} \sum_{t \in \mathcal{T}'_{s,[\hat{\kappa}_{\underline{q},T,n})}} \upsilon_{s,r,t,\hat{\kappa}_{\underline{q},T,n}} x_{s,t,\mu_{t,n}} \smallplus \sum_{q=\underline{q}}^{\overline{q}} \sum_{j \in \mathcal{J}} \sum_{t \in \mathcal{T}'_{j,[T)}} \Gamma_{j,r,q,t,T,n} \ x_{j,t,\mu_{t,n}} \smallge \sum_{q=\underline{q}}^{\overline{q}} o_{r,q,T,n}
    & \underline{q} \in \mathcal{Q}, \overline{q} \in  \{\underline{q}, \ldots, Q\}\label{eq:validineq7}
\end{flalign}
}

On the demand side, summing the demand satisfaction constraints~\eqref{eq:demandConstr} over $q \in \{\underline{q},\ldots,\overline{q}\}$ provides~\eqref{eq:validineq8}:

{\footnotesize
\begin{flalign}
    & \sum_{q=\underline{q}}^{\overline{q}} \left( o_{r,q,T,n} \smallplus  \sum_{a \in \mathcal{A}} \sum_{t \in \mathcal{T}'_{a,[T)}} \mathbbm{1}_{a,r,q,t,T,n} \ m_{a,q,t,T,n} \right) \smallge \sum_{q=\underline{q}}^{\overline{q}} \mathit{\delta}_{r,q,T}
    & \underline{q} \in \mathcal{Q}, \overline{q} \in  \{\underline{q}, \ldots, Q\}\label{eq:validineq8}
\end{flalign}
}

Since the coefficients $\mathbbm{1}_{a,r,q,t,T,n}$ are negative when $r = \mathrm{electricity}$, omitting these terms preserves the validity of the inequalities. To formalize this, we introduce new coefficients $\varsigma_{a,r,t,T}$ (defined in~\eqref{eq:varsigmaint}) that retain only the nonnegative heat coefficients. Substituting these new coefficients and bounding the heat transfer variables by their capacity constraints~\eqref{eq:heattransferConstr} yields the simplified inequalities~\eqref{eq:validineq9}:

{\footnotesize
\begin{flalign}
    & \sum_{q=\underline{q}}^{\overline{q}} \left( o_{r,q,T,n} \smallplus  \sum_{a \in \mathcal{A}} \sum_{t \in \mathcal{T}'_{a,[T)}} \varsigma_{a,r,t,T} \ x_{a,t,\mu_{t,n}} \right) \smallge \sum_{q=\underline{q}}^{\overline{q}} \mathit{\delta}_{r,q,T}
    & \underline{q} \in \mathcal{Q}, \overline{q} \in  \{\underline{q}, \ldots, Q\}\label{eq:validineq9}
\end{flalign}
}

Finally, combining~\eqref{eq:validineq7} and~\eqref{eq:validineq9} yields~\eqref{eq:validineq10}, showing that the family of inequalities given in~\eqref{eq:validineq10} is valid.
\end{proof}

Inequalities in~\eqref{eq:validineq10} admit a direct physical interpretation. For any contiguous block of sub-periods in the final period, the aggregate demand within that block must be covered by the maximum energy that can be supplied by the installed resources. This amount consists of three components: (i) the usable energy recoverable from installed storage capacity, under the most favorable assumption that the storage system enters the block fully charged; (ii) the energy generated by installed technologies during the block; and, for heat demand only, (iii) the energy deliverable through installed heat-transfer capacity. Hence, these inequalities are necessary conditions for operational feasibility expressed solely in terms of the investment decision variables. Although inequalities in~\eqref{eq:validineq10} are implied by~\eqref{eq:stochasticModel} and therefore redundant for the original formulation, they strengthen the master problem   in the Benders decomposition algorithm described in Section~\ref{subsec:benders-decomposition}, which only contains   investment decision variables. Section~\ref{subsec:separation-valid-inequalities} describes how these inequalities are used as cutting planes within the proposed algorithm.

\section{Uncertainty Modeling}
\label{sec:uncertainty-modeling}

This section describes how uncertainty is characterized and embedded in the proposed framework. We distinguish two categories: Technological uncertainty concerns the long-term evolution of technology installation costs and efficiencies over the planning horizon while operational uncertainty concerns the short-term variability of energy demand and technology performance within each operating period. Technological uncertainty determines the stochastic parameters in the MSSP model, while operational uncertainty is incorporated into the robust model extension and the out-of-sample Monte Carlo simulations.

\subsection{Technological Uncertainty}

We build the scenario tree in two steps, following our earlier work~\citep{csener2025dynamic}. First, we cluster historical cost-reduction and efficiency-improvement rates to estimate technological advancement multipliers (Section~\ref{subsubsec:adv-multipliers}). Second, we propagate these multipliers from an initial deterministic state to obtain the cost and performance parameters at every node of the scenario tree (Section~\ref{subsection:construct_scenariotree}).

\subsubsection{Technological Advancement Multipliers}
\label{subsubsec:adv-multipliers}

Motivated by Moore’s Law, which characterizes technological performance improvement as exponential over time~\citep{Moore1998}, we model technological progress through logarithmic improvement rates. Let $\nu_{u}^t \in \mathbb{R}_+^2$ denote the cost and efficiency of technology $u \in \mathcal{U}$ in year $t \in \{1,\dots,T\}$. For a stage length $\tau$, the $\tau$-year logarithmic improvement rates for cost ($\xi_{u,1}^t$) and efficiency ($\xi_{u,2}^t$) are defined as:

{\small
\begin{flalign}
    & \xi_{u,1}^t := \ln\left(\frac{\nu^{t}_{u,1}}{\nu^{t+\tau}_{u,1}}\right), \quad \xi_{u,2}^t := \ln\left(\frac{\nu^{t+\tau}_{u,2}}{\nu^{t}_{u,2}}\right) & t=1,\dots,T\smallminus\tau. \notag
\end{flalign}
}

To capture the empirical distribution of these improvements, we partition the observed rate vectors $\xi_u^t = (\xi_{u,1}^t, \xi_{u,2}^t)^\top \in \mathbb{R}_+^2$ into $|K_u|$ clusters. We determine the optimal cluster representatives $w_u^k \in \mathbb{R}_+^2$ and the binary assignment variables $z_{u,t,k} \in \{0,1\}$ by solving a mixed-integer nonlinear program (MINLP) that minimizes the within-cluster sum of squared $\ell_2$-distances:

{\small
\begin{flalign}
    & \min_{w_u^k \in \mathbb{R}_+^2, \ z_{u,t,k} \in \{0,1\}} \left\{ \sum_{t=1}^{T\smallminus\tau} \sum_{k \in K_u} z_{u,t,k} \| w_u^k \smallminus \xi_u^t \|_2^2 : \sum_{k \in K_u} z_{u,t,k} = 1, \ t=1,\dots,T\smallminus\tau \right\}. 
 \notag
\end{flalign}
}

Let $\tilde{w}_u^k$ and $\tilde{z}_{u,t,k}$ be the optimal solutions to this MINLP. We define the $\tau$-year improvement multipliers for each cluster $k$ as $\Xi_{u,1}^k := \exp(\smallminus\tilde{w}_{u,1}^k)$ and $\Xi_{u,2}^k := \exp(\tilde{w}_{u,2}^k)$. We then assume the technological parameters evolve according to the multiplier vector $\Xi_u^k = (\Xi_{u,1}^k, \Xi_{u,2}^k)^\top$ with probability $\chi_u^k := \frac{1}{T\smallminus\tau}\sum_{t=1}^{T\smallminus\tau} \tilde{z}_{u,t,k}$. This probability represents the proportion of observations assigned to cluster $k \in K_u$.

\subsubsection{Scenario Tree Construction}
\label{subsection:construct_scenariotree}

Using the multipliers derived above, we construct the scenario tree following the standard finite, rooted-tree representation of the underlying stochastic process~\citep{dupavcova2000scenarios}. An independent tree is generated for each technology $u \in \mathcal{U}$, where $T_u = (N(T_u), E(T_u))$ is a full, balanced tree with $\sum_{i=1}^{|\Psi|} |K_u|^{i \smallminus 1}$ nodes. Each node $n \in N(T_u)$ carries three attributes: an installation cost, a node probability $\pi_n$, and a technology performance vector (e.g., energy generation output or coefficient of performance).

To illustrate the tree construction, consider an energy generation technology $j \in \mathcal{J}$. The root node $n_{\mathrm{root}} \in N(T_j)$ represents the deterministic initial state at stage 1 (i.e., $\psi_{n_{\mathrm{root}}} = 1$), where the initial installation cost $\alpha_{j,1}$ and the initial energy generation profile $(\Gamma_{j,r,q,t,t,1})_{t \in \mathcal{T}_1}$ are known, deterministic parameters. The root node probability is initialized as $\pi_{n_{\mathrm{root}}} = 1$.

Each non-leaf node branches into $|K_j|$ child nodes, each corresponding to the realization of a distinct cluster $k \in K_j$. For any non-root node $n \in N(T_j) \setminus \{n_{\mathrm{root}}\}$, let $p(n)$ denote its immediate parent and $k_n$ the cluster realization governing the transition from $p(n)$ to $n$. The edge set is then defined as $E(T_j) = \{(p(n), n) : n \in N(T_j) \setminus \{n_{\mathrm{root}}\}\}$. The installation cost, node probability, and generation vector are then computed recursively as:

{\small
\begin{flalign}
    & \alpha_{j,n} = \alpha_{j,p(n)} \ \Xi_{j,1}^{k_n}, \quad \pi_n = \pi_{p(n)} \ \chi_j^{k_n} & j \in \mathcal{J}, n \in N(T_j) \setminus \{n_{\mathrm{root}}\} \notag \\
    & \Gamma_{j,r,q,t,t,n} = \Gamma_{j,r,q,\bar{t},\bar{t},p(n)} \ \Xi_{j,2}^{k_n} & j \in \mathcal{J}, r \in \mathcal{R}, q \in \mathcal{Q}, n \in N(T_j) \setminus \{n_{\mathrm{root}}\}, t \in \mathcal{T}_n, \bar{t} \in \mathcal{T}_{p(n)}. \notag
\end{flalign}
}

For any period $t' > t$, the generation profile is further adjusted by a degradation factor $\sigma_{j,t,t'}$:

{\small
\begin{flalign}
    & \Gamma_{j,r,q,t,t',n} = \Gamma_{j,r,q,t,t,n} \cdot (1 - \sigma_{j,t,t'}) 
    & j \in \mathcal{J}, r \in \mathcal{R}, q \in \mathcal{Q}, n \in N(T_j), t \in \mathcal{T}_n, t' \in \mathcal{T}_{j,[t)}. \notag
\end{flalign}
}

Once the individual trees $T_u$ have been built for all $u \in \mathcal{U}$, they are assembled into a unified scenario tree $T = (N(T), E(T))$. Let $N_s(T_u) := \{n_u \in N(T_u) : \psi_{n_u} = s\}$ denote the set of nodes in tree $T_u$ belonging to stage $s \in \Psi$. Since the technologies are assumed to evolve independently, the node set at any stage $s$ is the Cartesian product of the per-technology node sets $N_s(T_u)$ across all $u \in \mathcal{U}$. We further introduce a deterministic initialization node (node~$0$), the parent of $n_{\mathrm{root}}$, which encodes the technologies already in place prior to the first decision stage and has probability $\pi_{0} = 1$. Consequently, each node $n \in N(T)$ takes the form of a tuple $n = (n_u)_{u \in \mathcal{U}}$, capturing the joint realization of all technologies at that stage. The full node and edge sets of $T$ are thus:

{\small
\begin{flalign}
    & N(T) = \left\{0\right\} \cup \bigcup_{s \in \Psi} \prod_{u \in \mathcal{U}} N_s(T_u), \quad E(T) = \left\{(0, n_{\mathrm{root}})\right\} \cup \left\{ (m, n) \in N(T) \times N(T) : (m_u, n_u) \in E(T_u), \ u \in \mathcal{U} \right\}. \notag
\end{flalign}
}

The unconditional probability of each node $n \in N(T)$ is obtained as the product of its component node probabilities, $\pi_n = \prod_{u \in \mathcal{U}} \pi_{n_u}$. The attributes encoded in the resulting scenario tree constitute the stochastic parameters of the MSSP model given in~\eqref{eq:stochasticModel}.

\subsection{Operational Uncertainty}
\label{subsec:operational-uncertain-data}

Energy demand, renewable generation, and the coefficient of performance of heat-transfer technologies are inherently uncertain, driven by exogenous factors such as weather conditions and human activity~\citep{ZAKARIA20201543}.
We characterize these uncertainties through probability distributions derived from historical operational data via time-series clustering---a method commonly applied in clean energy transition planning to identify representative days~\citep{LI2022107697, TEICHGRAEBER20191283}.
To this end, we divide a given historical time series into consecutive daily profiles. Let $\mathcal{D}$ denote the set of days and $v_{d,i}$ the observation for day $d \in \mathcal{D}$ at hour $i \in \{1,\ldots,24\}$. To cluster daily profiles based on their temporal shape rather than absolute magnitude, we standardize each profile using its mean $\bar{v}_{d} = \frac{1}{24} \sum_{i=1}^{24} v_{d,i}$ and standard deviation $s_{d}$, where $s_{d}^{2} = \frac{1}{24} \sum_{i=1}^{24} (v_{d,i} - \bar{v}_{d})^{2}$. The standardized profile is then $g_{d,i} = (v_{d,i} - \bar{v}_{d})/s_{d}$.

Given these standardized daily profiles, we partition them into clusters. Since no clustering method is universally superior~\citep{LI2022107697, TEICHGRAEBER20191283}, we evaluate several approaches---including $k$-means and hierarchical clustering under various parameter settings---and select the configuration with the lowest intra-cluster variance. This configuration is agglomerative hierarchical clustering with Euclidean distance and Ward's minimum-variance linkage criterion, restricted to at most five clusters.

Let $\mathcal{C}$ denote the set of clusters, and for each $c \in \mathcal{C}$, let $\mathcal{D}_c \subseteq \mathcal{D}$ denote the subset of days assigned to it. The hourly cluster mean and variance are then computed as $\bar{g}_{c,i} = \frac{1}{|\mathcal{D}_c|}\sum_{d \in \mathcal{D}_c} g_{d,i}$ and $s_{c,i}^{2} = \frac{1}{|\mathcal{D}_c|}\sum_{d \in \mathcal{D}_c} (g_{d,i} - \bar{g}_{c,i})^{2}$. Figure~\ref{fig:cluster-profiles} illustrates the resulting clusters for the 2023 hourly electricity demand series of the case study introduced in Section~\ref{subsec:experimental-setup}, together with the representative profile of each cluster (dashed lines). The representative profile of a cluster is the member profile with the minimum total Euclidean distance to all others in the cluster.
These cluster statistics define the probability distributions used to represent operational uncertainty. For each cluster $c \in \mathcal{C}$, day $d \in \mathcal{D}_c$, and hour $i$, we sample the standardized value as $\tilde{g}_{d,i} \sim \mathrm{N}(g_{d,i}, s_{c,i}^{2})$, introducing cluster variability at each hour around each observed value. Transforming back to the original scale gives $\tilde{v}_{d,i} = s_d \tilde{g}_{d,i} + \bar{v}_d \sim \mathrm{N}(v_{d,i}, (s_d s_{c,i})^{2})$. Each resulting distribution is therefore centered on the observed value, with variance jointly governed by the within-day standard deviation of the original profile and the hour-specific variability of the corresponding cluster. Applying this procedure independently to each series, we obtain the probability distributions for the demand parameters $\delta_{r,q,t}$, the generation parameters $\Gamma_{j,r,q,t,t',n}$, and the coefficients of performance $\mathrm{cop}_{a,q,t,n}$.

\begin{figure}[H]
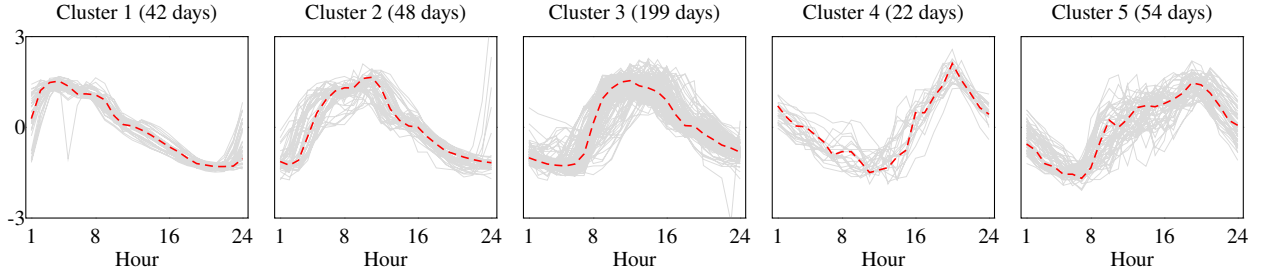

  \pgfplotsset{
    clusteraxis/.style={
      scale only axis, width=\clusteraxisw, height=2.4cm,
      unbounded coords=jump,
      xmin=0.4, xmax=24.6, xtick={1,8,16,24},
      ymin=-3, ymax=3, ytick={-3, 0, 3},
      major tick length=0pt,
      every tick label/.append style={inner xsep=0pt},
      xlabel={Hour}, xlabel style={yshift=0.15cm},
      ylabel={}, yticklabels={},
      tick label style={font=\scriptsize},
      label style={font=\scriptsize},
      title style={font=\scriptsize},
      every axis title shift=1pt,
      every axis plot/.append style={mark=none},
    }
  }
  \centering
  \newsavebox{\clustercalbox}
  \newlength{\clusteraxisw}
  \setlength{\clusteraxisw}{3cm}
  \sbox{\clustercalbox}{

  \caption{Clusters of standardized daily profiles of the 2023 hourly electricity demand series for the case study introduced in Section~\ref{subsec:experimental-setup}. Dashed lines show the representative profile of each cluster.}
  \label{fig:cluster-profiles}
\end{figure}

\section{Methodology}
\label{sec:methodology-benders}

In this section, we present our solution approach for the clean energy transition planning problem introduced in Section~\ref{sec:modelling-framework}. We begin by recasting the model in a compact form and introducing a scenario-wise reformulation, which motivates the proposed decomposition scheme. Section~\ref{subsec:benders-decomposition} then develops the Benders decomposition algorithm with partial non-anticipativity relaxation, and Section~\ref{subsec:rolling-horizon-sim} describes the rolling-horizon simulation used to evaluate the resulting investment plans in an out-of-sample analysis.

Let $x$ and $y$ denote the vectors of investment and operational decision variables, respectively. Let $c$ and $d$ represent their corresponding objective coefficient vectors. We define the constraint sets that do not couple investment and operational decision variables respectively as $\mathcal{X} := \left\{ x : \eqref{eq:budgetConstr}, \eqref{eq:spatialConstr}, \eqref{eq:initializationConstr}, \eqref{eq:VarDomain1} \right\}$ and $\mathcal{Y} := \left\{ y : \eqref{eq:demandConstr}, \eqref{eq:storagebalanceConstr}, \eqref{eq:emissionConstr}, \eqref{eq:VarDomain2}, \eqref{eq:VarDomain3} \right\}$. The constraints coupling these decisions---namely~\eqref{eq:demandConstr2},~\eqref{eq:storagecapacityConstr}, and~\eqref{eq:heattransferConstr}---are expressed generically as $Ax + By \ge b$. Then, the model~\eqref{eq:stochasticModel} can be provided in compact form as
\begin{equation}
\label{eq:compact-model-form}
\min_{\substack{x \in \mathcal{X}, y \in \mathcal{Y}}}
\left\{
c^{\top} x \smallplus d^{\top} y : Ax + By \ge b
\right\}.
\end{equation}

We now provide a scenario-wise reformulation of the model with partial non-anticipativity relaxation. Let $\Omega$ denote the set of scenario paths in the scenario tree, and let $\Pi_\omega$ denote the probability of path $\omega \in \Omega$. For each path, let $x_\omega$ and $y_\omega$ denote the associated investment and operational decision vectors, with $c_\omega$ and $d_\omega$ the corresponding cost coefficient vectors (excluding node probabilities), and let $A_\omega$, $B_\omega$, and $b_\omega$ denote the path-specific constraint parameters. In a standard MSSP formulation, non-anticipativity constraints require that any two scenario paths passing through the same node take identical decisions at that node. We relax these constraints for the operational variables while preserving them for the investment variables. The relaxation renders the operational feasible region separable across paths, yielding path-specific sets $\mathcal{Y}_\omega$. Then the relaxed problem becomes:

\begin{equation}
\label{eq:relaxed-mssp}
\min_{\substack{(x_\omega, y_\omega)_{\omega \in \Omega}}}
\left\{
\sum_{\omega \in \Omega} \Pi_\omega (c_{\omega}^{\top} \ x_\omega
\smallplus d_{\omega}^{\top} \ y_\omega) :
\begin{array}{ll}
A_\omega \, x_\omega \smallplus B_\omega \, y_\omega \smallge b_\omega, & \omega \in \Omega \\
y_\omega \in \mathcal{Y}_\omega, & \omega \in \Omega \\
(x_\omega)_{\omega \in \Omega} \in \mathcal{X}
\end{array}
\right\}.
\end{equation}

Because non-anticipativity is no longer enforced on the operational variables, an optimal solution to~\eqref{eq:relaxed-mssp} may assign different operational decisions to scenario paths that share the same node. Once an optimal solution to~\eqref{eq:relaxed-mssp} is obtained, we restore non-anticipativity through the procedure described in Section~\ref{subsec:restore-non-anticipativity}.

\subsection{Benders Decomposition with Partial Non-anticipativity Relaxation}
\label{subsec:benders-decomposition}

For a fixed investment decision vector $(x_\omega^\ast)_{\omega \in \Omega}$, problem~\eqref{eq:relaxed-mssp} decomposes by scenario path into $|\Omega|$ independent subproblems. For each scenario path $\omega \in \Omega$, the corresponding subproblem is given by

\begin{flalign}
\label{eq:subproblem_primal}
& \boldsymbol{SP}_\omega(x_\omega^\ast) :
\min_{y_\omega \in \mathcal{Y}_\omega}
\left\{
d_\omega^\top \, y_\omega
\;:\;
B_\omega \, y_\omega \smallge b_\omega \smallminus A_\omega \, x_\omega^\ast
\right\}, & \omega \in \Omega.
\end{flalign}

To approximate the objective value of each scenario-path subproblem within the master problem, we introduce auxiliary variables $(\theta_\omega)_{\omega \in \Omega}$ and progressively refine this approximation with Benders cuts generated over a sequence of iterations indexed by $k$. At iteration $k$, the scenario-path subproblems $\boldsymbol{SP}_\omega(x_\omega^{(k)})$ are solved at the current candidate investment vector $(x_\omega^{(k)})_{\omega \in \Omega}$. A feasible subproblem yields an optimal dual extreme point $\lambda_\omega^{(k)}$ that generates an optimality cut, whereas an infeasible subproblem yields a dual extreme ray $\mu_\omega^{(k)}$ that generates a feasibility cut. A newly generated point or ray is retained only if its associated cut is violated by the current candidate solution. Let $\Lambda^{\mathrm{opt}}_\omega$ and $\Lambda^{\mathrm{feas}}_\omega$ denote the sets of dual extreme points and extreme rays retained up to iteration $k$, respectively. The associated set of Benders cuts is then given by

\begin{equation}
\label{eq:benderscutset}
C \coloneqq
\left\{
\begin{aligned}
& \theta_\omega \smallge (\lambda_\omega)^\top \bigl( b_\omega \smallminus A_\omega \, x_\omega \bigr),
&& \omega \in \Omega, \; \lambda_\omega \in \Lambda^{\mathrm{opt}}_\omega, \\[2pt]
& 0 \smallge (\mu_\omega)^\top \bigl( b_\omega \smallminus A_\omega \, x_\omega \bigr),
&& \omega \in \Omega, \; \mu_\omega \in \Lambda^{\mathrm{feas}}_\omega
\end{aligned}
\right\}.
\end{equation}

Each element of $C$ is a linear inequality in the master variables $(x_\omega, \theta_\omega)_{\omega \in \Omega}$. Using this cut set, the master problem is defined as

\begin{equation}
\label{eq:master_problem}
\boldsymbol{MP}(C) :
\min_{(x_\omega,\, \theta_\omega)_{\omega \in \Omega}}
\left\{
\sum_{\omega \in \Omega} \Pi_\omega
\bigl( c_\omega^\top x_\omega \smallplus \theta_\omega \bigr)
\;:\;
(x_\omega)_{\omega \in \Omega} \in \mathcal{X}, \;
(x_\omega, \theta_\omega)_{\omega \in \Omega} \text{ satisfies } C
\right\}.
\end{equation}

Algorithm~\ref{alg:benders-decomposition} summarizes the proposed solution procedure, in which the scenario-path subproblems are solved in parallel at each iteration. In addition to the Benders cuts generated from these subproblems, the master problem is strengthened by valid inequalities: the \textsc{SeparationKadane} procedure described in Section~\ref{subsec:separation-valid-inequalities} identifies violated valid inequalities (line~\ref{line:valid-inequality}), which are then added to the master problem together with the Benders cuts.

The procedure is organized into two phases that differ in how the integrality constraints and the cut generation are handled. In Phase~1, the integrality constraints of the master problem are relaxed, so that only its linear relaxation is solved at each Benders iteration, and both feasibility and optimality cuts are added at every iteration to accelerate convergence. The algorithm switches to Phase~2 once the convergence of the relaxed master problem slows down, that is, when the Benders optimality tolerance is attained by the linear relaxation or the lower bound stagnates over a window of consecutive iterations. Phase~2 reinstates the integrality constraints; both cut types are added in its early iterations, whereas in the later iterations feasibility cuts are prioritized and optimality cuts are added only when all subproblems are feasible, so that the search concentrates on restoring subproblem feasibility. To limit the growth of the master problem, we remove cuts that are no longer binding: a cut whose slack ranks in the highest $50\%$ of all cut slacks for $100$ consecutive iterations is moved from \(C\) to a pool (line~\ref{line:cut-management}), and any cut in the pool that is violated by the current master solution is reintroduced to \(C\). The algorithm terminates when the gap between the best lower bound from the master problem and the incumbent objective value falls within the prescribed Benders optimality tolerance $(\varepsilon_B)$, after which a final step restores non-anticipativity of the operational decisions (line~\ref{line:non-anticipativity-correction}), as described in Section~\ref{subsec:restore-non-anticipativity}.

\begin{algorithm}[H]
\footnotesize
\caption{Benders decomposition framework for the MSSP problem}
\label{alg:benders-decomposition}
\begin{algorithmic}[1]
\Require Set of scenario paths $\Omega$ with their probabilities $(\Pi_\omega)_{\omega \in \Omega}$; model inputs $\mathcal{X}$ and $(c_\omega, d_\omega, b_\omega, A_\omega, B_\omega, \mathcal{Y}_\omega)_{\omega \in \Omega}$; functions $V_{\omega,r} : \{(x_\omega,\underline{q},\overline{q}) \in \mathcal{X}_\omega \!\times\! \mathcal{Q}^2 \mid \underline{q} \le \overline{q}\} \to \mathbb{R}$ for all $\omega \in \Omega$ and all $r \in \mathcal{R}$; Benders tolerance $\varepsilon_{B} \smallge 0$, violation tolerance $\varepsilon_{V} \smallge 0$.
\Ensure Solution $x^\ast, y^\ast$ and objective value $\mathrm{UB}$.

\State $\mathrm{LB}^{(0)} \gets - \infty$, \ $\mathrm{UB} \gets + \infty$, \ $k \gets 1$, \ $C \gets \emptyset$, \ $\mathrm{Phase} \gets 1$

\State Relax the integrality constraints of $\boldsymbol{MP}(C)$.

\While{$\mathrm{Phase} \smalleq 1$ \textbf{or} $|\mathrm{UB} - \mathrm{LB}^{(k-1)}| / |\mathrm{UB}| > \varepsilon_{B}$}

    \State Solve problem $\boldsymbol{MP}(C)$ and let $(x^{(k)},\theta^{(k)})$ and $z_{lb}^{(k)}$ denote the optimal solution and lower bound, respectively.

    \State $C_{\mathrm{feas}} \gets \emptyset$, \ $C_{\mathrm{opt}} \gets \emptyset$

    \For{\textbf{all} $\omega \in \Omega$ \textbf{in parallel}}
    
    \State Solve subproblem $\boldsymbol{SP}_\omega(x_{\omega}^{(k)})$.

    \If{$\boldsymbol{SP}_\omega(x_{\omega}^{(k)})$ is infeasible}

    \State Let $\mu^{(k)}_\omega$ denote a dual extreme ray associated with $\boldsymbol{SP}_\omega(x_{\omega}^{(k)})$, $z_\omega \gets +\infty$.

    \If{$(\mu_\omega^{(k)})^\top (b_\omega \smallminus A_{\omega} \ x_{\omega}^{(k)}) \smallge \varepsilon_{V}$}

    \State $C_{\mathrm{feas}} \gets C_{\mathrm{feas}} \cup \{ 0 \smallge (\mu_\omega^{(k)})^\top (b_\omega \smallminus A_{\omega} \ x_{\omega}) \}$

    \EndIf

    \State $C_{\mathrm{feas}} \gets C_{\mathrm{feas}} \cup \Call{SeparationKadane}{V_{\omega,r}, x_\omega^{(k)}, \mathcal{R}, \mathcal{Q}}$  \label{line:valid-inequality}

    \Else
    \State Let $\lambda_\omega^{(k)}$ and $z_\omega$ denote an optimal dual extreme point and the optimal objective value of $\boldsymbol{SP}_\omega(x_{\omega}^{(k)})$, respectively.

    \If{$(\lambda_\omega^{(k)})^\top (b_\omega \smallminus A_{\omega} \ x_{\omega}^{(k)}) \smallminus \theta_\omega^{(k)} \smallge \varepsilon_{V}$}

    \State $C_{\mathrm{opt}} \gets C_{\mathrm{opt}} \cup \{ \theta_\omega \smallge (\lambda_\omega^{(k)})^\top (b_\omega \smallminus A_{\omega} \  x_{\omega}) \}$

    \EndIf

    \EndIf

    \EndFor

    \If{$\mathrm{Phase} \smalleq 1$ \textbf{or} in the first 200 Phase 2 iterations \textbf{or} an incumbent was found within the last 25 iterations \textbf{or} $C_{\mathrm{feas}} \smalleq \emptyset$}

    \State $C \gets C \cup C_{\mathrm{opt}} \cup C_{\mathrm{feas}}$

    \Else

    \State $C \gets C \cup C_{\mathrm{feas}}$

    \EndIf

    \State $z^{(k)} \gets \sum_{\omega \in \Omega} \Pi_\omega \left(c_\omega^\top \ x^{(k)}_\omega + z_\omega \right)$
    \If{$z^{(k)} < \mathrm{UB}$}
        \State $\mathrm{UB} \gets z^{(k)}$, \  $x^\ast \gets x^{(k)}$
    \EndIf

    \State Update \(C\) by pruning selected cuts and reinstating violated cuts from the pruned cuts pool. \label{line:cut-management}

    \State $\mathrm{LB}^{(k)} \gets \max\{z_{lb}^{(k)}, \mathrm{LB}^{(k-1)}\}$

    \If{$\mathrm{Phase} \smalleq 1$ \textbf{and} $(|\mathrm{UB} - \mathrm{LB}^{(k)}| / |\mathrm{UB}| \smallle \varepsilon_B \ \textbf{or} \ |\mathrm{LB}^{(k)} - \mathrm{LB}^{(\max\{k\smallminus50, 0\})}| / |\mathrm{LB}^{(k)}| \smallle 10^{\smallminus 3})$} \label{line:phase-transition-start}

    \State Reintroduce the integrality constraints of $\boldsymbol{MP}(C)$, $\mathrm{Phase} \gets 2$, \ $\mathrm{UB} \gets + \infty$. \label{line:phase-transition-end}

    \EndIf

    \State $k \gets k + 1$ 

\EndWhile

\State Let $y^\ast$ and $z_{\mathrm{na}}$ be the optimal solution and objective value of the correction LP described in Section~\ref{subsec:restore-non-anticipativity}. \label{line:non-anticipativity-correction}

\State \Return $x^\ast, y^\ast, z_{\mathrm{na}} \smallplus \sum_{\omega \in \Omega} \Pi_\omega \, c_\omega^\top \ x^\ast_\omega$
\end{algorithmic}
\end{algorithm}

\subsubsection{Separation of Valid Inequalities}
\label{subsec:separation-valid-inequalities}

To accelerate the convergence of the Benders decomposition algorithm, we incorporate the valid inequalities in~\eqref{eq:validineq10} in the master problem. By projecting necessary operational-feasibility conditions onto the investment space, these inequalities enable the master problem to anticipate operational constraints. These cuts thereby tighten the feasible region of the master problem by eliminating operationally infeasible investment decisions, reducing the number of Benders iterations required for convergence.

Adding all valid inequalities to the master problem 
is computationally impractical, as one such inequality is defined for each combination of a scenario path, an energy category, and a sub-period pair. To avoid incorporating this large set of inequalities upfront, we employ a separation procedure within the Benders algorithm. At each iteration, given the current master problem solution~$x^{(k)}$, we solve a separation problem for each energy category and {infeasible} subproblem to identify the most violated inequalities, which are then added to the master problem alongside the Benders cuts.

Let $\mathcal{X}_\omega$ denote the projection of $\mathcal{X}$ onto the investment variables on the nodes of path $\omega$. For a given scenario path $\omega \in \Omega$ and energy category $r \in \mathcal{R}$, we quantify the violation of inequality~\eqref{eq:validineq10} at the current master problem solution~$x^{(k)}_\omega \in \mathcal{X}_\omega$ using the function:

{\footnotesize
\begin{flalign}
& V_{\omega,r}(x^{(k)}_\omega,\underline{q},\overline{q})
:= \eta^{\smallminus}_{r} \eta_{r} \sum_{s \in \mathcal{S}} \sum_{t \in \mathcal{T}'_{s,[\hat{\kappa}_{\underline{q},T,n_\omega})}} \upsilon_{s,r,t,\hat{\kappa}_{\underline{q},T,n_\omega}} x^{(k)}_{s,t,\mu_{t,n_\omega}} \smallplus \sum_{q=\underline{q}}^{\overline{q}} \sum_{j \in \mathcal{J}} \sum_{t \in \mathcal{T}'_{j,[T)}} \Gamma_{j,r,q,t,T,n_\omega} x^{(k)}_{j,t,\mu_{t,n_\omega}} & \notag \\ 
& \qquad \qquad \qquad \qquad \quad \smallplus \sum_{q=\underline{q}}^{\overline{q}}  \sum_{a \in \mathcal{A}} \sum_{t \in \mathcal{T}'_{a,[T)}} \varsigma_{a,r,t,T} x^{(k)}_{a,t,\mu_{t,n_\omega}} \smallminus \sum_{q=\underline{q}}^{\overline{q}} \mathit{\delta}_{r,q,T}
\end{flalign}
}

\noindent where $n_\omega$ denotes the leaf node associated with scenario path $\omega$. The separation problem then seeks the sub-period pair that maximizes this violation (i.e., minimizes the function value):

{\footnotesize
\begin{flalign}
    \min_{1\le\underline{q}\le\overline{q}\le Q} V_{\omega,r}(x^{(k)}_\omega,\underline{q},\overline{q}) \label{eq:violationminimization}
\end{flalign}
}

If the optimal value of~\eqref{eq:violationminimization} is strictly negative, the 
corresponding inequality~\eqref{eq:validineq10} is violated at the current solution 
and therefore added to the master problem as a  cutting plane. Otherwise, 
no violated inequality of this form exists for the given scenario path and energy 
category, and no cut is generated.

To solve~\eqref{eq:violationminimization}, we construct an array of length~$|\mathcal{Q}|$ whose $q$-th entry stores the value $V_{\omega,r}(x^{(k)}_\omega,q,q)$, the violation associated with sub-period $q \in \mathcal{Q}$. Identifying the most violated sub-period pair $(\underline{q},\overline{q})$ then reduces to finding a contiguous subarray with the minimum sum. As established in Proposition~\ref{prop:kadane}, Kadane's algorithm~\citep{bentley1984programming} solves this problem in linear time.

\begin{proposition}
\label{prop:kadane}
For a sequence $A = (a_1, \dots, a_n) \in \mathbb{R}^n$, Kadane's algorithm identifies a contiguous subarray with the minimum sum in $\mathcal{O}(n)$ time.
\end{proposition}

\begin{proof}
Correctness and linear time complexity follow directly from the standard recurrence relation underlying Kadane's algorithm; for a detailed discussion, see \citet{gries1982note}.
\end{proof}

For a given scenario path, solving~\eqref{eq:violationminimization} across all energy categories and separating the most violated inequality requires $\mathcal{O}(|\mathcal{R}|  \cdot |\mathcal{Q}|)$ time. Algorithm~\ref{alg:separationkadane} details the complete procedure and generates the corresponding valid cuts.

\begin{algorithm}[H]
\caption{Separation of valid inequalities using Kadane's algorithm}
\footnotesize
\label{alg:separationkadane}
\begin{algorithmic}[1]
\Require Functions $V_{\omega,r} : \{(x_\omega,\underline{q},\overline{q}) \in \mathcal{X}_\omega \!\times\! \mathcal{Q}^2 \mid \underline{q} \le \overline{q}\} \to \mathbb{R}$ for $\forall r \in \mathcal{R}$ with fixed $\omega \in \Omega$, $x_\omega^{(k)}$, $\mathcal{R}$, $\mathcal{Q}$, and tolerance $\varepsilon_{V} \smallge 0$.
\Ensure Set of violated valid inequalities $C_v$.

\Function{SeparationKadane}{$V_{\omega,r}, x_\omega^{(k)}, \mathcal{R}, \mathcal{Q}$}

\State $C_v \gets \emptyset$

\ForAll{$r \in \mathcal{R}$}
    \State $\text{currentSum} \gets V_{\omega,r}(x_\omega^{(k)},1,1), \ \underline{q}_{\mathrm{curr}} \gets 1$ \Comment{\footnotesize Local minimum sum ending at $q$ and its start index}
    
    \State $\text{bestSum} \gets \text{currentSum}$, \ $\underline{q}^* \gets 1$, \ $\overline{q}^* \gets 1$ \Comment{\footnotesize Global minimum sum over $[1, q]$ and its boundaries}
    
    \ForAll{$q \in \{2,\ldots,Q\}$}
        \If{$\text{currentSum} > 0$}
            \State $\text{currentSum} \gets V_{\omega,r}(x_\omega^{(k)},q,q)$, \ $\underline{q}_{\mathrm{curr}} \gets q$ \Comment{\footnotesize Initiate a new locally optimal subarray starting at $q$}
        \Else
            \State $\text{currentSum} \gets \text{currentSum} + V_{\omega,r}(x_\omega^{(k)},q,q)$ \Comment{\footnotesize Extend the previous subarray}
        \EndIf
    
        \If{$\text{currentSum} < \text{bestSum}$}
            \State $\text{bestSum} \gets \text{currentSum}$, \ $\underline{q}^* \gets \underline{q}_{\mathrm{curr}}$, \ $\overline{q}^* \gets q$ \Comment{\footnotesize Update global minimum sum}
        \EndIf
    \EndFor
    \If{$\text{bestSum} < \smallminus \varepsilon_{V}$}
    \State $C_v \gets C_v \cup \{V_{\omega,r}(x_\omega,\underline{q}^*,\overline{q}^*) \ge 0\}$
    \EndIf
\EndFor

\State \Return $C_v$

\EndFunction
\end{algorithmic}
\end{algorithm}

\subsubsection{Restoring Non-anticipativity}
\label{subsec:restore-non-anticipativity}

Upon termination, Algorithm~\ref{alg:benders-decomposition} returns an incumbent $(x^\ast,y^\ast)$ of the relaxed problem~\eqref{eq:relaxed-mssp}, where $x^\ast$ is non-anticipative by construction but $y^\ast$ need not be. We restore non-anticipativity in a correction step that fixes $x$ at $x^\ast$ and re-optimizes $y$ by solving an LP subject to the constraints of the original model~\eqref{eq:stochasticModel} and the non-anticipativity constraints. We first show that, under a mild assumption on the stage structure, such an incumbent always admits a non-anticipative solution at $x^\ast$ that is feasible for the original model and whose objective function value exceeds that of $(x^\ast,y^\ast)$ by a bounded amount. To this end, let $\underline t_n$ and $\overline t_n$ denote the first and last periods of node $n \in \mathcal{N}$, and let $\overline \mu_n$ denote the set of immediate children of $n$ in the scenario tree.

\begin{proposition}
\label{prop:na-feasible}
Let $(x_\omega^\ast,y_\omega^\ast)_{\omega\in\Omega}$ be an incumbent of the non-anticipativity-relaxed problem~\eqref{eq:relaxed-mssp}, and assume that each node in the final stage consists of more than one period. Then, there exists a non-anticipative solution $(x_\omega^\ast,\hat y_\omega)_{\omega\in\Omega}$ of~\eqref{eq:stochasticModel} whose objective function value exceeds that of $(x_\omega^\ast,y_\omega^\ast)_{\omega\in\Omega}$ by at most
\begin{equation}
\label{eq:na-increment-bound}
\sum_{r \in \mathcal{R}} \sum_{n \in \mathcal{N}\setminus\mathcal{N}_l} \sum_{n' \in \overline \mu_n} \frac{\pi_{n'}\, \beta^{(\underline{t}_{n'}-1)}\, \lambda_{r,\underline{t}_{n'}}}{\eta^{\smallplus}_{r}} \bigl(\Delta_{r,n,n'}\bigr)^{\!+},
\end{equation}
where $(\cdot)^{+} := \max\{0,\cdot\}$ and $\Delta_{r,n,n'}$ is the residual of the inter-node storage-balance constraints~\eqref{eq:storagebalanceConstr}, evaluated at the node-wise solution $\tilde y$ obtained when these constraints are relaxed:

\begin{equation}
\label{eq:na-residual}
\Delta_{r,n,n'} := \tilde h_{r,1,\underline{t}_{n'},n'} \smallminus \eta_{r}\, \tilde h_{r,Q,\overline{t}_n,n} \smallminus \eta^{\smallplus}_{r}\, \tilde h_{r,1,\underline{t}_{n'},n'}^{\smallplus} \smallplus \frac{\tilde h_{r,1,\underline{t}_{n'},n'}^{\smallminus}}{\eta^{\smallminus}_{r}}.
\end{equation}
\end{proposition}

\begin{proof}
Within $\boldsymbol{SP}_\omega(x_\omega^\ast)$, consecutive nodes along path $\omega$ are coupled through the storage-balance constraints~\eqref{eq:storagebalanceConstr}---one per energy category---which relate the terminal sub-period of each node to the first sub-period of its successor. Relaxing these inter-node constraints decomposes $\boldsymbol{SP}_\omega(x_\omega^\ast)$ into node-wise subproblems $\boldsymbol{\widetilde{SP}}_{\omega,i}(x_\omega^\ast)$, one per stage $i \in \Psi$, while retaining the within-node storage balances.

Consider two scenario paths $\omega_1,\omega_2\in\Omega$ that pass through a common node $n \in \mathcal{N}$. The node-wise subproblems $\boldsymbol{\widetilde{SP}}_{\omega_1, \psi_n}(x_{\omega_1}^\ast)$ and $\boldsymbol{\widetilde{SP}}_{\omega_2, \psi_n}(x_{\omega_2}^\ast)$ have identical feasible regions and objective coefficients: their investment decisions agree because $x^\ast$ is non-anticipative on all shared nodes, and their stochastic parameters agree because both are evaluated at $n$. The node-wise subproblem at $n$ is therefore independent of the path containing $n$, so we may discard the duplicates and associate a single subproblem $\boldsymbol{\widetilde{SP}}_{n}(x_n)$ with each node $n \in \mathcal{N}$, where $x_n$ is the common investment decision at $n$ and its ancestors. Any non-anticipativity violation in $y^\ast$ at a shared node is thus attributable solely to the relaxed inter-node storage-balance constraints.

Each $\boldsymbol{\widetilde{SP}}_{n}(x_n)$ is feasible: it is obtained by relaxing the inter-node constraints of $\boldsymbol{SP}_\omega(x_\omega^\ast)$, which is itself feasible because $(x_\omega^\ast,y_\omega^\ast)_{\omega\in\Omega}$ is feasible for~\eqref{eq:relaxed-mssp}. Let $\tilde y$ denote the non-anticipative solution obtained by solving these node-wise subproblems; it satisfies every constraint of~\eqref{eq:stochasticModel} except, possibly, the relaxed inter-node storage-balance constraints. Because each $\boldsymbol{\widetilde{SP}}_{n}(x_n)$ relaxes the inter-node coupling of the corresponding per-path subproblem, the restriction of $y^\ast$ to node $n$ is feasible for $\boldsymbol{\widetilde{SP}}_{n}(x_n)$; the node-wise optimum at $n$ therefore costs no more than $y^\ast$ does at $n$, and summing over all nodes shows that the objective of $\tilde y$ is no greater than that of $y^\ast$.

It remains to reinstate the inter-node constraints while preserving non-anticipativity and feasibility. Consider a non-leaf node $n \in \mathcal{N}\setminus\mathcal{N}_l$ and a child $n' \in \overline \mu_n$. The reinstated constraint couples the terminal storage $h_{r,Q,\overline{t}_n,n}$ of $n$ to the first sub-period of $n'$ through
\begin{equation}
h_{r,1,\underline{t}_{n'},n'} = \eta_{r}\, h_{r,Q,\overline{t}_n,n} \smallplus \eta^{\smallplus}_{r}\, h_{r,1,\underline{t}_{n'},n'}^{\smallplus} \smallminus \frac{h_{r,1,\underline{t}_{n'},n'}^{\smallminus}}{\eta^{\smallminus}_{r}}.
\end{equation}
Evaluated at the node-wise solution $\tilde y$, this equation need not hold; its residual $\Delta_{r,n,n'}$ is defined as in~\eqref{eq:na-residual}. A positive residual indicates that the first sub-period of $n'$ draws more stored energy than its inputs supply. We restore the balance at $n'$ itself, leaving the storage carried over from $n$ untouched: we raise the charging from $\tilde h_{r,1,\underline{t}_{n'},n'}^{\smallplus}$ to $\hat h_{r,1,\underline{t}_{n'},n'}^{\smallplus} := \tilde h_{r,1,\underline{t}_{n'},n'}^{\smallplus} \smallplus \Delta_{r,n,n'}/\eta^{\smallplus}_{r}$ and meet it with an equal increase in the exogenous purchase, $\hat p_{r,1,\underline{t}_{n'},n'} := \tilde p_{r,1,\underline{t}_{n'},n'} \smallplus \Delta_{r,n,n'}/\eta^{\smallplus}_{r}$, in the first sub-period of $n'$. Because the purchase and the charging rise by the same amount, the demand constraints~\eqref{eq:demandConstr2} are unaffected; and because the stored level $\hat h_{r,1,\underline{t}_{n'},n'} = \tilde h_{r,1,\underline{t}_{n'},n'}$ is left unchanged, so is the remainder of the storage trajectory at $n'$. A negative residual is handled symmetrically by raising the discharge $\hat h_{r,1,\underline{t}_{n'},n'}^{\smallminus}$ above $\tilde h_{r,1,\underline{t}_{n'},n'}^{\smallminus}$, which disposes of the surplus carried-over energy and requires no purchase. Every remaining component of $\hat y$ coincides with that of $\tilde y$.

This correction preserves feasibility. No storage level is increased anywhere in the tree, so the storage-capacity constraints~\eqref{eq:storagecapacityConstr} remain satisfied. Moreover, every corrective purchase lies in the first sub-period of a node: because each node in the final stage consists of more than one period, the first period $\underline{t}_{n'}$ of every node $n'$---including each leaf node---satisfies $\underline{t}_{n'} < T$. Since the emission constraints~\eqref{eq:emissionConstr} restrict exogenous purchases only in the final period $T$ of the leaf nodes, these corrective purchases are unrestricted. The corrected solution $\hat y$ is therefore non-anticipative and feasible for~\eqref{eq:stochasticModel}.

The corrected solution $\hat y$ differs from the node-wise solution $\tilde y$ only in these additional purchases, which offset the positive residuals. Restoring the balance at a child $n'$ raises the purchase by $\bigl(\Delta_{r,n,n'}\bigr)^{+}/\eta^{\smallplus}_{r}$, with the charging efficiency $\eta^{\smallplus}_{r}$ relating the purchased energy to the amount actually stored. Weighting each such purchase by its discounted unit cost $\pi_{n'}\,\beta^{(\underline{t}_{n'}-1)}\,\lambda_{r,\underline{t}_{n'}}$ and summing over all energy categories, non-leaf nodes, and their children bounds the objective increase of $\hat y$ over $\tilde y$ by~\eqref{eq:na-increment-bound}. Because the objective of $\tilde y$ is no greater than that of $y^\ast$, the objective of $\hat y$ exceeds that of $(x_\omega^\ast,y_\omega^\ast)_{\omega\in\Omega}$ by at most~\eqref{eq:na-increment-bound}, as claimed.
\end{proof}

Rather than solving the node-wise problems and carrying out the construction described above, we recover non-anticipativity by solving a single LP. Because leaf nodes are not shared across scenario paths, the incumbent operational decisions $y^\ast$ are already non-anticipative there, and only the internal nodes can exhibit violations. The LP fixes the investments at $x^\ast$ and the leaf-node operational decisions at $y^\ast$---except those in each leaf's first sub-period, which links the leaf to its parent node---and re-optimizes the remaining operational variables subject to the constraints of~\eqref{eq:stochasticModel} and non-anticipativity. Because the non-anticipative solution constructed in Proposition~\ref{prop:na-feasible} is feasible for this LP, the LP attains an objective that exceeds the incumbent's by at most~\eqref{eq:na-increment-bound}, a finite quantity since the storage installations---fixed at $x^\ast$ and bounded in~\eqref{eq:VarDomain1}---bound the storage variables through the storage-capacity constraints~\eqref{eq:storagecapacityConstr}. In Section~\ref{subsec:comp-performance}, we  empirically observe that the correction step does not increase the objective value in any instance.

\subsection{Out-of-Sample Performance Evaluation via Rolling-Horizon Simulation}
\label{subsec:rolling-horizon-sim}

The optimization model~\eqref{eq:stochasticModel} is solved with nominal profiles for energy demand, renewable generation, and the heat pump coefficient of performance. In practice, these quantities may deviate from their nominal values. To assess how the resulting investment plans perform under such operational uncertainty, we conduct Monte Carlo simulations with realizations sampled from the distributions described in Section~\ref{subsec:operational-uncertain-data}. Each realization is generated by adding either white- or pink-noise deviations to the nominal profiles, with pink noise producing temporally correlated deviations~\citep{Guevara2022}. Given an investment plan, the dispatch decisions within a single replication are governed by the following problem:

\begin{equation}
\label{eq:monte-carlo-opt}
\text{minimize} \left\{ \sum_{n \in \mathcal{N}\ssetminussobj \zeroobj} \sum_{t \in \mathcal{T}_n} \sum_{r \in \mathcal{R}} \sum_{q \in \mathcal{Q}} \pi_n \beta^{(t-1)} \lambda_{r,t} \ p_{r,q,t,n} : \eqref{eq:demandConstr}, \eqref{eq:demandConstr2}, \eqref{eq:storagebalanceConstr}, \eqref{eq:storagecapacityConstr}, \eqref{eq:heattransferConstr}, \eqref{eq:VarDomain2}, \eqref{eq:VarDomain3} \right\}.
\end{equation}

\noindent
Here,  parameters $\Gamma_{j,r,q,t,t',n}$, $\mathbbm{1}_{a,r,q,t,t',n}$, and $\mathit{\delta}_{r,q,t'}$ in constraints~\eqref{eq:demandConstr} and~\eqref{eq:demandConstr2} are set to their sampled values. A complete simulation comprises many such replications, each under an independently sampled~realization.

Solving problem~\eqref{eq:monte-carlo-opt} over the full horizon at once would assume perfect foresight of every future realization, whereas operational decisions are in fact made sequentially as uncertainty is revealed. To respect this information structure, we evaluate each replication through a sequence of LP dispatch problems solved over rolling windows that together span the planning horizon. Each window covers $\hat{l} \in \mathbb{Z}_{+}$ consecutive sub-periods: the first $\bar{l} \in \mathbb{Z}_{+}$ of these, with $\bar{l} \smallle \hat{l}$, take their sampled realizations, while the remaining sub-periods retain their nominal values and serve as a planning lookahead. We solve the window's dispatch problem, record the decisions of its first sub-period, and advance the window by one sub-period. The procedure repeats until the window has traversed every node and sub-period of the scenario tree.

To break ties among alternative optima, we assign a small negative objective coefficient to the stored-energy variables $h_{r,q,t,n}$ in the final sub-period of each window; this rewards carrying forward as much stored energy as possible without altering the primary objective. Emission constraints are omitted from the dispatch problems; any resulting net-zero shortfall is instead recorded and reported as the percentage of final-period demand met by exogenous purchases.

In the dispatch problems described so far, the lookahead sub-periods take nominal values. To make the dispatch decisions themselves robust to unfavorable conditions over the lookahead, we additionally construct a robust dispatch problem that mirrors the robust reformulation of Section~\ref{subsec:robust-reformulation}. Within each window, the first~$\bar{l}$ sub-periods retain their sampled realizations, while the lookahead sub-periods $(\bar{l} \smallplus 1, \ldots, \hat{l})$ are displaced to the adverse end of their uncertainty intervals at robustness level $\epsilon$. This provides a dispatch-side hedge that complements the robustness already embedded in the investment plans. In the Monte Carlo simulations, we evaluate both the nominal and the robust investment plans, dispatching each under both the nominal and the robust dispatch problems.

\section{Computational Study}
\label{sec:computations}

This section details the computational study of the proposed framework. Section~\ref{subsec:experimental-setup} describes the campus-scale instance and the computational environment, Section~\ref{subsec:comp-performance} evaluates the computational performance against the benchmark methods, and Section~\ref{subsec:case-study} presents the case study for a university campus.

\subsection{Experimental Setup}
\label{subsec:experimental-setup}

All computational experiments are conducted on a 64-bit Apple Silicon MacBook Pro running macOS Tahoe 26.2, equipped with an Apple M3 Pro processor (6 performance and 6 efficiency cores) and 36 GB of unified memory. The algorithms are implemented in Python 3.13.6, and all optimization problems are solved using Gurobi Optimizer 13.0.2 via its Python interface. The source code developed for this study is publicly available online~\citep{stochastic-campus-energy-transition}.

The experiments are based on a clean electricity--heat transition planning instance constructed for the METU campus. In accordance with METU's net-zero target for 2040, the study spans the 2026--2040 horizon. Exogenous energy procurement is therefore prohibited in the final period by setting $\gamma_{r} \smalleq 0$ for $r \in \mathcal{R}$. The planning horizon of $T \smalleq 15$ years is divided equally into $|\Psi| \smalleq 3$ stages of five years (periods), and each year is partitioned into sub-periods that aggregate two hours, yielding $|\mathcal{Q}| \smalleq 4{,}368$ sub-periods per year. An annual budget of \$20 million (M) is imposed in all periods.

Figure~\ref{fig:energy-system} presents a schematic of the integrated electricity--heat system formed by the candidate technologies considered for the transition. Photovoltaic panels and wind turbines are the electricity generation technologies, whereas parabolic troughs provide heat generation. Heat pumps serve as heat-transfer technologies that deliver heat using electricity, and their conversion efficiency is characterized by a coefficient of performance defined for each sub-period. Energy is stored by lithium-ion batteries on the electricity side and by hot water tanks on the heat side. Among these $|\mathcal{U}| \smalleq 6$ technologies, photovoltaic panels and lithium-ion batteries are subject to stochastic technological advancement, each with two advancement clusters, while the remaining four are treated as either mature or deterministically progressing. Two versions of photovoltaic panels with different capacities are considered, both following the same technological progress. Table~\ref{tab:VersionsTable} details the parameters of these candidate technologies.

\sankeyset{
  new end style={bigarrow}{
    ([xshift=-0.2mm]\name.left) -- ([xshift=-0.2mm,yshift=1.6mm]\name.left)
    -- ([xshift=3.5mm]\name.center) -- ([xshift=-0.2mm,yshift=-1.6mm]\name.right)
    -- ([xshift=-0.2mm]\name.right) -- cycle
  }{
    ([xshift=-0.2mm]\name.left) -- ([xshift=-0.2mm,yshift=1.6mm]\name.left)
    -- ([xshift=3.5mm]\name.center) -- ([xshift=-0.2mm,yshift=-1.6mm]\name.right)
    -- ([xshift=-0.2mm]\name.right)
  },
}

\definecolor{elecGold}{RGB}{242,183,5}
\definecolor{heatRed}{RGB}{224,58,46}
\definecolor{windBlue}{RGB}{45,130,210}
\definecolor{solarOrange}{RGB}{240,140,30}

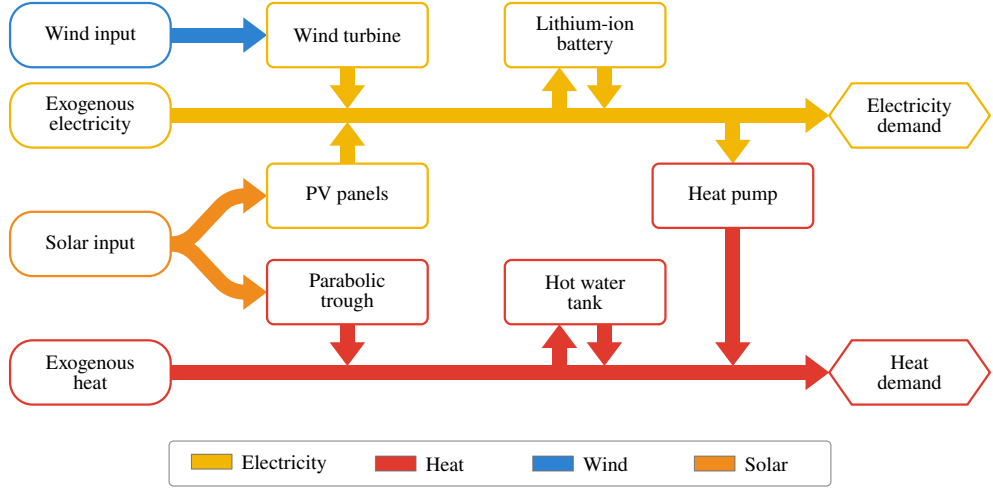
\begin{figure}[H]
\centering
\scalebox{0.85}{
\begin{tikzpicture}[
font=\footnotesize,
tech/.style = {draw, rounded corners=3pt, line width=1pt, fill=white,
               minimum width=2.5cm, minimum height=1cm, align=center,
               inner sep=2pt, font=\footnotesize\linespread{0.8}\selectfont},
output/.style = {signal, signal to=east and west, signal pointer angle=110,
               draw, line width=1pt, fill=white,
               minimum width=2.5cm, minimum height=1cm, align=center,
               inner sep=2pt, font=\footnotesize\linespread{0.8}\selectfont},
input/.style = {draw, rounded corners=10pt, line width=1pt, fill=white,
               minimum width=2.5cm, minimum height=1cm, align=center,
               inner sep=2pt, font=\footnotesize\linespread{0.8}\selectfont}
]
\begin{sankeydiagram}[start style=simple, end style=bigarrow, draw/.style={draw=none}, ratio=1cm/3]

\sankeyset{fill/.style={fill=elecGold}}
\sankeynodestart{name=a,at={4,4.75},angle=270,quantity=0.7}
\sankeyadvance{a}{0.3cm}
\sankeynodeend{as=a}

\sankeyset{fill/.style={fill=elecGold}}
\sankeynodestart{name=a,at={4,3.25},angle=90,quantity=0.7}
\sankeyadvance{a}{0.3cm}
\sankeynodeend{as=a}

\sankeyset{fill/.style={fill=elecGold}}
\sankeynodestart{name=a,at={7.3,4},angle=90,quantity=0.7}
\sankeyadvance{a}{0.4cm}
\sankeynodeend{as=a}

\sankeyset{fill/.style={fill=elecGold}}
\sankeynodestart{name=a,at={8,4.75},angle=270,quantity=0.7}
\sankeyadvance{a}{0.3cm}
\sankeynodeend{as=a}

\sankeyset{fill/.style={fill=heatRed}}
\sankeynodestart{name=a,at={7.3,0},angle=90,quantity=0.7}
\sankeyadvance{a}{0.4cm}
\sankeynodeend{as=a}

\sankeyset{fill/.style={fill=heatRed}}
\sankeynodestart{name=a,at={8,0.75},angle=270,quantity=0.7}
\sankeyadvance{a}{0.3cm}
\sankeynodeend{as=a}

\sankeyset{fill/.style={fill=elecGold}}
\sankeynodestart{name=a,at={10,4},angle=270,quantity=0.7}
\sankeyadvance{a}{0.4cm}
\sankeynodeend{as=a}

\sankeyset{fill/.style={fill=heatRed}}
\sankeynodestart{name=a,at={10,2.25},angle=270,quantity=0.7}
\sankeyadvance{a}{1.8cm}
\sankeynodeend{as=a}

\sankeyset{fill/.style={fill=solarOrange}}
\sankeynodestart{name=a,at={1.25,2},quantity=0.7}
\sankeynodeend{name=b,at={2.4,2.75},quantity=0.7}
\sankeydubins[minimum radius=4mm]{a}{b}

\sankeyset{fill/.style={fill=solarOrange}}
\sankeynodestart{name=a,at={1.25,2},quantity=0.7}
\sankeynodeend{name=b,at={2.4,1.25},quantity=0.7}
\sankeydubins[minimum radius=4mm]{a}{b}

\sankeyset{fill/.style={fill=heatRed}}
\sankeynodestart{name=a,at={4,0.75},angle=270,quantity=0.7}
\sankeyadvance{a}{0.3cm}
\sankeynodeend{as=a}

\sankeyset{fill/.style={fill=windBlue}}
\sankeynodestart{name=a,at={1.25,5.25},quantity=0.7}
\sankeyadvance{a}{1.15cm}
\sankeynodeend{as=a}

\sankeyset{fill/.style={fill=elecGold}}
\sankeynodestart{name=a,at={1.25,4},quantity=0.7}
\sankeyadvance{a}{9.9cm}
\sankeynodeend{as=a}

\sankeyset{fill/.style={fill=heatRed}}
\sankeynodestart{name=a,at={1.25,0},quantity=0.7}
\sankeyadvance{a}{9.9cm}
\sankeynodeend{as=a}

\node[input, draw=windBlue]    at (0,5.25)   {Wind input};
\node[input, draw=elecGold]    at (0,4)      {Exogenous \\ electricity};
\node[input, draw=solarOrange] at (0,2)      {Solar input};
\node[input, draw=heatRed]     at (0,0)      {Exogenous \\ heat};
\node[tech, draw=elecGold]    at (4,5.25)   {Wind turbine};
\node[tech, draw=elecGold]    at (4,2.75)   {PV panels};
\node[tech, draw=heatRed]     at (4,1.25)   {Parabolic \\ trough};
\node[tech, draw=elecGold]    at (7.7,5.25) {Lithium-ion \\ battery};
\node[tech, draw=heatRed]     at (7.7,1.25) {Hot water \\ tank};
\node[tech, draw=heatRed]     at (10,2.75)  {Heat pump};
\node[output, draw=elecGold]  at (12.75,4)  {Electricity \\ demand};
\node[output, draw=heatRed]   at (12.75,0)  {Heat \\ demand};
\end{sankeydiagram}

\begin{scope}[font=\footnotesize]
  \path (current bounding box.south) ++(0,-0.90) coordinate (legend-center);
  \begin{scope}[shift={(legend-center)}]
    \draw[rounded corners=2pt, draw=black!45, line width=0.5pt]
      (-5.15,0.35) rectangle (5.15,-0.35);
    \foreach \x/\legendcolor/\legendlabel in {
      -3.75/elecGold/Electricity,
      -1.25/heatRed/Heat,
       1.25/windBlue/Wind,
       3.75/solarOrange/Solar
    }{
      \node[anchor=center, inner sep=0pt, minimum width=2.2cm] at (\x,0) {%
        {\setlength{\fboxsep}{0pt}\setlength{\fboxrule}{0.3pt}%
        \fcolorbox{black!55}{\legendcolor}{\makebox[0.62cm]{\rule{0pt}{0.22cm}}}}%
        \hspace{0.15cm}\legendlabel%
      };
    }
  \end{scope}
\end{scope}
\end{tikzpicture}
}
\caption{Schematic of the integrated electricity--heat system.}
\label{fig:energy-system}
\end{figure}

\vspace{-0.2cm}
The resulting scenario tree comprises $|\mathcal{N}| \smalleq 22$ nodes and $|\Omega| \smalleq 16$ scenario paths. In total, the corresponding optimization problem contains 9,435,001 constraints and 10,352,904 decision variables, of which 742 are investment variables and 318 are integer. The complete dataset used in the experiments is detailed in \ref{sec:appendixA}.

\subsection{Computational Performance}
\label{subsec:comp-performance}

We evaluate the computational performance of the solution methodology proposed in Section~\ref{subsec:benders-decomposition} on the instance described in Section~\ref{subsec:experimental-setup} at four operational temporal resolutions, defined by sub-periods that aggregate 24, 8, 4, and 2 hours. Each instance is solved to a relative optimality gap of 1\%, subject to a time limit of 4 hours.

\subsubsection{Linear Programming Relaxation}
\label{subsubsec:lp-relaxation}

To demonstrate the effect of our algorithmic design choices on solution time, we first solve the LP relaxation of model~\eqref{eq:stochasticModel} with different methods. Table~\ref{tab:lp-relaxation-results} summarizes the performance of four approaches applied to the LP relaxation of model~\eqref{eq:stochasticModel} across the considered temporal resolutions. For each method, it reports the best lower and upper bounds ($\mathrm{LB}$ and $\mathrm{UB}$) on the objective value in millions of dollars, together with the resulting relative optimality gap (Gap) as a percentage. Solution times are given in seconds: a single total time (Time) for the \textit{Extensive Form}, and the cumulative master-problem (Mstr.), subproblem (Sub.), and total (Total) times for each decomposition method. The four methods are defined as follows:  \textit{Extensive Form} solves the LP relaxation directly. \textit{Algorithm~\ref{alg:benders-decomposition}} is the full decomposition algorithm, including both the valid inequalities (VIs) and the two-phase cut-addition strategy. \textit{Algorithm~\ref{alg:benders-decomposition} without VIs} differs from \textit{Algorithm~\ref{alg:benders-decomposition}} only in that line~\ref{line:valid-inequality} is removed, so that no valid inequalities are added to the master problem. \textit{Single-Phase Algorithm~\ref{alg:benders-decomposition}} is obtained from \textit{Algorithm~\ref{alg:benders-decomposition}} by removing lines~\ref{line:phase-transition-start}--\ref{line:phase-transition-end}, so that feasibility and optimality cuts are added together at every iteration without a phase transition. When Algorithm~\ref{alg:benders-decomposition} and its variants are applied to the LP relaxation, the integrality constraints are not reintroduced at the transition to Phase~2, so the two phases differ only in their cut-addition strategy.

\begin{table}[H]
\centering
\setlength{\tabcolsep}{2.5pt}
\renewcommand{\arraystretch}{1.3}
\caption{Comparison of solution methods for the LP relaxation of model~\eqref{eq:stochasticModel} across temporal resolutions}
\label{tab:lp-relaxation-results}
\resizebox{\textwidth}{!}{
\begin{tabular}{l rrrr @{\hskip 8pt} rrrrrr @{\hskip 8pt} rrrrrr @{\hskip 8pt} rrrrrr}
\toprule
& \multicolumn{4}{c@{\hskip 8pt}}{\textbf{Extensive Form}}
& \multicolumn{6}{c@{\hskip 8pt}}{\textbf{Single-Phase Algorithm~\ref{alg:benders-decomposition}}}
& \multicolumn{6}{c@{\hskip 8pt}}{\textbf{Algorithm~\ref{alg:benders-decomposition} without VIs}}
& \multicolumn{6}{c}{\textbf{Algorithm~\ref{alg:benders-decomposition}}} \\
\cmidrule(lr){2-5}
\cmidrule(lr){6-11}
\cmidrule(lr){12-17}
\cmidrule(lr){18-23}
& Time & $\mathrm{LB}$ & $\mathrm{UB}$ & Gap
& Mstr. & Sub. & Total & $\mathrm{LB}$ & $\mathrm{UB}$ & Gap
& Mstr. & Sub. & Total & $\mathrm{LB}$ & $\mathrm{UB}$ & Gap
& Mstr. & Sub. & Total & $\mathrm{LB}$ & $\mathrm{UB}$ & Gap \\
\midrule
24-h & 107 & 122.53 & 122.53 & 0.00 & 2 & 38 & 41 & 122.51 & 122.53 & 0.02 & 2 & 39 & 41 & 122.51 & 122.53 & 0.02 & 2 & 31 & 34 & 122.51 & 122.56 & 0.04 \\
8-h & 3,450 & 155.79 & 155.79 & 0.00 & 17 & 411 & 432 & 155.77 & 155.81 & 0.03 & 8 & 312 & 323 & 155.33 & 156.20 & 0.56 & 8 & 281 & 292 & 155.33 & 156.01 & 0.43 \\
4-h & $14,400^\dagger$ & {--} & {--} & {--} & 57 & 1,588 & 1,656 & 157.69 & 157.69 & 0.00 & 11 & 911 & 927 & 157.06 & 158.18 & 0.71 & 12 & 882 & 901 & 156.93 & 158.08 & 0.73 \\
2-h & $14,400^\dagger$ & {--} & {--} & {--} & 178 & 8,084 & 8,330 & 161.12 & 161.13 & 0.00 & 17 & 3,290 & 3,320 & 160.27 & 161.82 & 0.96 & 23 & 3,710 & 3,757 & 160.32 & 161.77 & 0.89 \\
\bottomrule
\end{tabular}
}
\vspace{2pt}
\begin{minipage}{\linewidth}
\noindent{\footnotesize $^\dagger$Time limit of 14{,}400 seconds (4 hours) reached. Dashes (--) indicate that no bound was available at termination.}
\end{minipage}
\end{table}

\vspace{-0.2cm}
Table~\ref{tab:lp-relaxation-results} shows that, on the LP relaxation, the \textit{Extensive Form} solves only the 24-hour and 8-hour instances to optimality, with solution times markedly longer than those of every decomposition-based variant. On the 4-hour and 2-hour instances, it reaches the time limit without returning any bound. All three variants of Algorithm~\ref{alg:benders-decomposition}, by contrast, solve every instance within the time limit. The proposed decomposition thus substantially outperforms the \textit{Extensive Form}, particularly at the finer resolutions.

Although \textit{Single-Phase Algorithm~\ref{alg:benders-decomposition}} employs the same 1\% optimality gap as a stopping criterion, it consistently terminates with much smaller gaps, stopping 
as soon as it obtains its first incumbent solution. This behavior reflects how upper bounds arise in Benders decomposition: a valid upper bound becomes available only when all scenario-path subproblems are feasible for a given master solution. Because the cuts generated by different subproblems share common master variables, adding feasibility cuts alone would steer the master problem toward solutions that render all subproblems feasible. In the single-phase variant, however, optimality cuts are added at every iteration alongside the feasibility cuts, and they also steer the master solution toward minimizing the objective approximation. The master solution is therefore perturbed more strongly and cannot concentrate on producing solutions that are feasible across all subproblems.

Since iterations at which all subproblems are feasible occur only rarely, the lower bound continues to tighten across the intervening iterations. By the time such an iteration is reached, the lower bound has already converged close to the optimal value, leaving a near-zero gap. In Phase~2 of the two-phase strategy, by contrast, only feasibility cuts are added, so the master problem is guided toward solutions that are feasible across all subproblems, and incumbents are found more frequently. The single-phase variant therefore attains tighter final gaps than the two-phase variants, but at the cost of substantially longer solution times. This underscores the value of the two-phase strategy: Phase~1 promotes rapid lower-bound convergence by adding feasibility and optimality cuts together whereas Phase~2 focuses exclusively on feasibility cuts to find more incumbents.

The contribution of the valid inequalities can be assessed by comparing \textit{Algorithm~\ref{alg:benders-decomposition}} against \textit{Algorithm~\ref{alg:benders-decomposition} without VIs}. Adding the valid inequalities reduces the total solution time at three of the four temporal resolutions, while the 2-hour instance exhibits the opposite behavior. On that instance, Phase~1 runs for 2{,}770 seconds and exits with a lower bound of \$159.6M when the valid inequalities are enabled, against 2{,}227 seconds and a weaker lower bound of \$156.9M when they are excluded. 

Although the valid inequalities accelerate lower-bound convergence, enabling them on the 2-hour instance delays the phase transition, causing the algorithm to spend more time in Phase~1. Once in Phase~2, the algorithm terminates in 978 seconds when the valid inequalities are enabled, compared with 1{,}088 seconds when they are disabled, consistent with the benefit of entering Phase~2 with a tighter lower bound. The longer Phase~1 nevertheless dominates, producing a higher total time. This outcome highlights the pivotal role of the phase-transition criterion (line~\ref{line:phase-transition-start}) in determining the overall solution time: when the transition occurs later, the additional time spent in Phase~1 can outweigh the subsequent savings in Phase~2.

\subsubsection{Mixed-Integer Linear Programming Model}
\label{subsubsec:milp-model}

Table~\ref{tab:benders-milp-results} compares the solution methods for the original MILP model~\eqref{eq:stochasticModel}. These methods mirror those used for the LP relaxation, except that \textit{Classical Benders} replaces the \textit{Single-Phase} variant. \textit{Classical Benders} applies Algorithm~\ref{alg:benders-decomposition} with the same master--subproblem partitioning of the decision variables as in~\eqref{eq:subproblem_primal} and~\eqref{eq:master_problem}, assigning the investment variables to the master problem and the operational variables to the subproblem. Unlike the methodology proposed in Section~\ref{sec:methodology-benders}, it does not apply the partial non-anticipativity relaxation, so the operational variables remain coupled across scenario paths and form a single subproblem rather than the scenario-wise subproblems of the proposed decomposition. The remaining three methods, together with the reported columns and their units, retain the definitions introduced for the LP relaxation.

\begin{table}[H]
\centering
\setlength{\tabcolsep}{2.5pt}
\renewcommand{\arraystretch}{1.3}
\caption{Comparison of solution methods for model~\eqref{eq:stochasticModel} across temporal resolutions}
\label{tab:benders-milp-results}
\resizebox{\textwidth}{!}{
\begin{tabular}{l rrrr @{\hskip 8pt} rrrrrr @{\hskip 8pt} rrrrrr @{\hskip 8pt} rrrrrr}
\toprule
& \multicolumn{4}{c@{\hskip 8pt}}{\textbf{Extensive Form}}
& \multicolumn{6}{c@{\hskip 8pt}}{\textbf{Classical Benders}}
& \multicolumn{6}{c@{\hskip 8pt}}{\textbf{Algorithm~\ref{alg:benders-decomposition} without VIs}}
& \multicolumn{6}{c}{\textbf{Algorithm~\ref{alg:benders-decomposition}}} \\
\cmidrule(lr){2-5}
\cmidrule(lr){6-11}
\cmidrule(lr){12-17}
\cmidrule(lr){18-23}
& Time & $\mathrm{LB}$ & $\mathrm{UB}$ & Gap
& Mstr. & Sub. & Total & $\mathrm{LB}$ & $\mathrm{UB}$ & Gap
& Mstr. & Sub. & Total & $\mathrm{LB}$ & $\mathrm{UB}$ & Gap
& Mstr. & Sub. & Total & $\mathrm{LB}$ & $\mathrm{UB}$ & Gap \\
\midrule
24-h & 1,063 & 122.65 & 123.76 & 0.90 & 1,912 & 1,897 & 3,821 & 122.10 & 123.33 & 1.00 & 41 & 44 & 86 & 122.56 & 123.12 & 0.45 & 73 & 40 & 115 & 122.56 & 123.15 & 0.48 \\
8-h & $14,400^\dagger$ & 155.97 & 158.89 & 1.84 & 2,490 & 11,878 & $14,400^\dagger$ & 151.85 & 159.30 & 4.68 & 367 & 394 & 767 & 155.46 & 156.78 & 0.84 & 266 & 376 & 650 & 155.65 & 156.73 & 0.69 \\
4-h & $14,400^\dagger$ & {--} & 310.57 & {--} & 914 & 13,476 & $14,400^\dagger$ & 128.08 & 167.33 & 23.46 & 385 & 1,256 & 1,654 & 157.36 & 158.94 & 1.00 & 357 & 1,255 & 1,628 & 157.30 & 158.81 & 0.95 \\
2-h & $14,400^\dagger$ & {--} & 312.07 & {--} & 144 & 14,254 & $14,400^\dagger$ & 113.74 & {--} & {--} & 1,842 & 4,813 & 6,699 & 160.67 & 162.25 & 0.97 & 1,144 & 4,760 & 5,956 & 160.67 & 162.27 & 0.99 \\
\bottomrule
\end{tabular}
}
\vspace{2pt}
\begin{minipage}{\linewidth}
\noindent{\footnotesize $^\dagger$Time limit of 14{,}400 seconds (4 hours) reached. Dashes (--) indicate that no bound was available at termination.}
\end{minipage}
\end{table}

\vspace{-0.2cm}
Table~\ref{tab:benders-milp-results} shows that, apart from the 24-hour instance, neither the \textit{Extensive Form} nor \textit{Classical Benders} reaches the 1\% optimality gap within the time limit. The \textit{Extensive Form} solves only the 24-hour instance within this gap and returns no lower bound at the 4-hour and 2-hour resolutions. \textit{Classical Benders} likewise solves only the 24-hour instance, where its solution time of 3{,}821 seconds exceeds even that of the \textit{Extensive Form}. At the finer resolutions, it reaches the time limit with gaps of 4.68\% or worse. Both variants of \textit{Algorithm~\ref{alg:benders-decomposition}}, by contrast, solve every instance within the 1\% optimality gap. On the 24-hour instance---the only one solved by either the \textit{Extensive Form} or \textit{Classical Benders}---they are roughly an order of magnitude faster than the \textit{Extensive Form} and more than thirty times faster than \textit{Classical Benders}.

Consistent with the LP-relaxation results, comparing \textit{Algorithm~\ref{alg:benders-decomposition}} against \textit{Algorithm~\ref{alg:benders-decomposition} without VIs} shows that the valid inequalities reduce the total solution time at the three finer resolutions. 
On the 2-hour instance, identifying the most violated valid inequalities and separating them with Algorithm~\ref{alg:separationkadane} takes 6.2 seconds over all iterations. At this finest resolution, \textit{Algorithm~\ref{alg:benders-decomposition}} terminates in 1{,}330 iterations, whereas \textit{Algorithm~\ref{alg:benders-decomposition} without VIs} requires 1{,}460. Given this negligible separation cost, the valid inequalities accelerate convergence and meaningfully reduce the overall solution time, confirming their contribution to the performance of the solution procedure. We therefore continue with \textit{Algorithm~\ref{alg:benders-decomposition}} in the remainder of the experiments.

When \textit{Algorithm~\ref{alg:benders-decomposition}} is applied to the 2-hour instance, the LP that corrects the non-anticipativity violations of the operational decision variables (line~\ref{line:non-anticipativity-correction}) takes 4.7 seconds to solve. This correction step does not increase the objective value in any instance solved by the non-anticipativity-relaxed methods, for either the LP or the MILP model. Although Proposition~\ref{prop:na-feasible} only provides an upper bound on this increment, in our experiments it is always zero. Overall, the proposed relaxation--correction strategy allows us to solve the original problem efficiently: relaxing non-anticipativity yields the decomposition that underlies the runtimes in Table~\ref{tab:benders-milp-results}, while restoring it afterward incurs only the negligible cost of the correction step.

\subsection{Case Study}
\label{subsec:case-study}

We now present the clean electricity--heat transition plan for the METU campus. We first describe the base-case investment plan, then evaluate the out-of-sample performance of the nominal and robust plans, and finally report sensitivity analyses on key problem parameters.

\subsubsection{Base Case}
\label{subsubsec:base-case}

Figure~\ref{fig:base-model-tree} presents the optimal investment plan obtained with \textit{Algorithm~\ref{alg:benders-decomposition}} on the instance with 2-hour operational resolution described in Section~\ref{subsec:experimental-setup}. Each node of the scenario tree reports the capacity installed for each technology in every year, and the root node holds the first-stage decisions that are common to all scenario paths. Each branch label indicates the pace of technological progress, $s$ for slow and $f$ for fast, with the first letter referring to photovoltaic panels and the second to lithium-ion batteries. For instance, scenario path $S_3$, denoted by $ss \smalltimes fs$, corresponds to a trajectory in which photovoltaic panels undergo slow progress followed by fast progress, while lithium-ion batteries undergo two consecutive stages of slow advancement.

The expected cost of the transition plan is \$162.27M, of which \$11.45M corresponds to electricity purchasing costs and \$23.64M to heat purchasing costs. The remaining \$127.18M consists of technology installation and O\&M costs. In the first stage, the plan installs 12.0~\si{\mega\watt} of solar, 12.0~\si{\mega\watt} of wind, 1.5~\si{\mega\watt\hour} of battery storage, and 4.4~\si{\mega\watt\hour} of heat-pump capacity, all in the first year. Renewable generation and storage capacities are then expanded substantially in the second and third stages, driven both by the progress in technology costs and efficiencies and by the net-zero emission target enforced in the final period. Across the planning horizon, photovoltaic panels and lithium-ion batteries account for most of the installed capacity, whereas wind capacity is expanded only sporadically. Photovoltaic panels are preferred over wind turbines because of the limited wind potential at the installation site, where wind turbines achieve an annual capacity factor of only 10.4\%, compared with 18.3\% for photovoltaic panels. Battery storage grows sharply toward the end of the horizon as renewable penetration increases, with annual installations of up to 539.2~\si{\mega\watt\hour}. Once grid procurement is prohibited in the final period, substantial storage is needed to balance the intermittent output of renewable generation with demand.

Heat demand, on the other hand, is met mostly through electrified heating. Heat-pump installations appear in every stage and increase markedly toward the end of the planning horizon across all scenario paths. Parabolic troughs and hot water tanks, by contrast, are deployed only in small quantities and at a limited number of nodes, indicating that thermal energy generation and thermal storage play a marginal role relative to heat pumps powered by clean electricity. This observation is consistent with parabolic troughs requiring further cost reductions to become economically competitive for widespread deployment~\citep{KHAN2024114551}.

\definecolor{solaryellow}{RGB}{255,240,181}
\definecolor{windcyan}{RGB}{209,237,242}
\definecolor{batterypurple}{RGB}{226,218,245}
\definecolor{parabolictroughbrown}{RGB}{255,224,200}
\definecolor{heatpumporange}{RGB}{230,234,219}
\definecolor{heatstoragepink}{RGB}{255,214,214}

\begin{figure}[H]
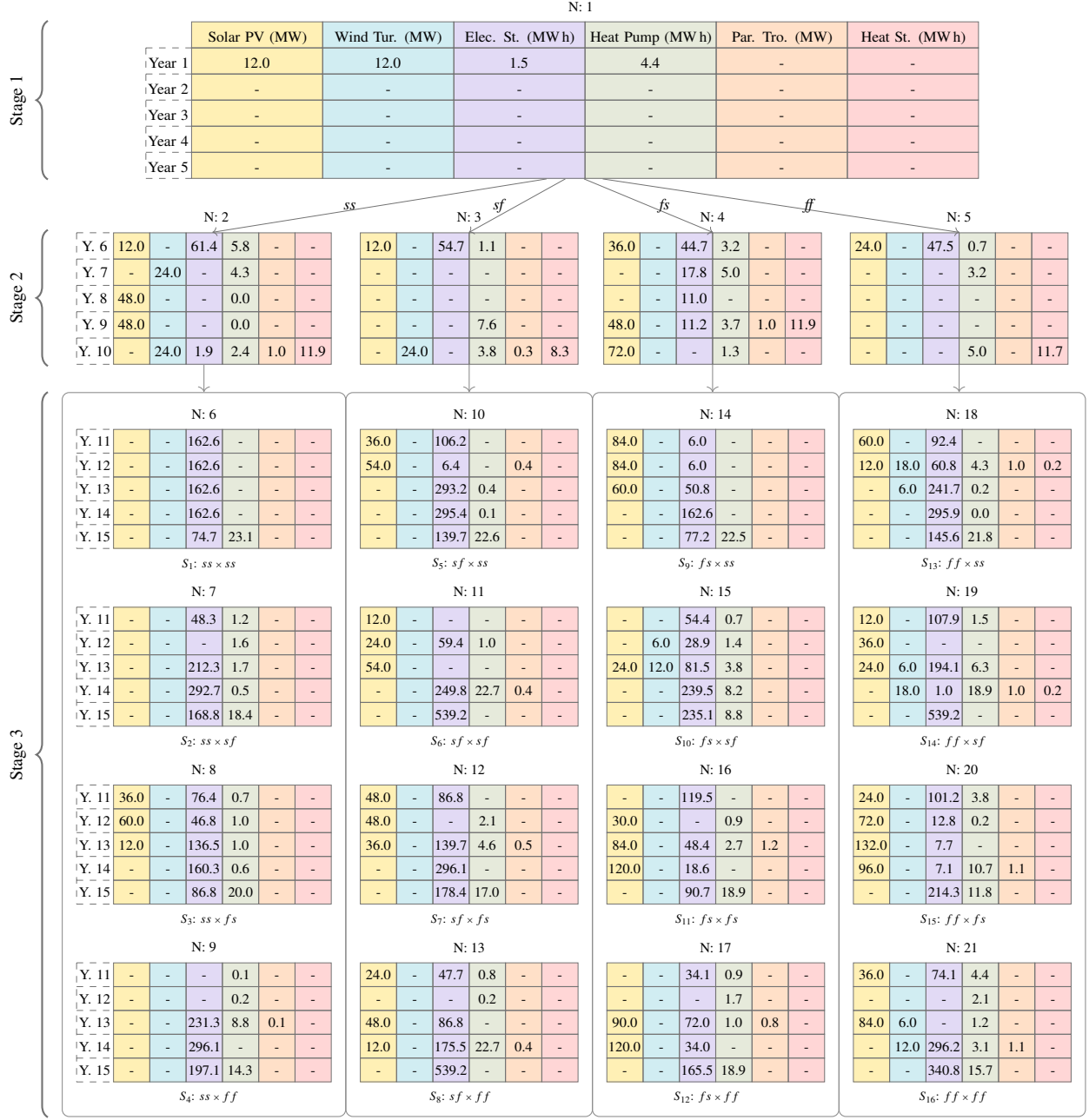

\centering
\begin{adjustbox}{max width=\linewidth, max height=0.93\textheight, keepaspectratio}

};
\draw ([xshift=-0.25cm]1.south) -- ([xshift=0.14cm]2.north) node [midway, left, shift={(-.25,0)}, edge label] {\fontsize{7.2}{8.2}\selectfont \textit{ss}};
\draw ([xshift=0.05cm]1.south) -- (3.north) node [midway, left, shift={(-.05,0)}, edge label] {\fontsize{7.2}{8.2}\selectfont \textit{sf}};
\draw ([xshift=0.35cm]1.south) -- (4.north) node [midway, right, shift={(.05,0)}, edge label] {\fontsize{7.2}{8.2}\selectfont \textit{fs}};
\draw ([xshift=0.65cm]1.south) -- (5.north) node [midway, right, shift={(.25,0)}, edge label] {\fontsize{7.2}{8.2}\selectfont \textit{ff}};
\node[above=1.0pt of 1, xshift=0.3cm] {\fontsize{6.4}{7.4}\selectfont N: 1};
\node[above=1.0pt of 2, xshift=0.2cm] {\fontsize{6.4}{7.4}\selectfont N: 2};
\node[above=1.0pt of 3] {\fontsize{6.4}{7.4}\selectfont N: 3};
\node[above=1.0pt of 4] {\fontsize{6.4}{7.4}\selectfont N: 4};
\node[above=1.0pt of 5] {\fontsize{6.4}{7.4}\selectfont N: 5};
\node[above=1.0pt of 6] {\fontsize{6.4}{7.4}\selectfont N: 6};
\node[above=1.0pt of 7] {\fontsize{6.4}{7.4}\selectfont N: 7};
\node[above=1.0pt of 8] {\fontsize{6.4}{7.4}\selectfont N: 8};
\node[above=1.0pt of 9] {\fontsize{6.4}{7.4}\selectfont N: 9};
\node[above=1.0pt of 10] {\fontsize{6.4}{7.4}\selectfont N: 10};
\node[above=1.0pt of 11] {\fontsize{6.4}{7.4}\selectfont N: 11};
\node[above=1.0pt of 12] {\fontsize{6.4}{7.4}\selectfont N: 12};
\node[above=1.0pt of 13] {\fontsize{6.4}{7.4}\selectfont N: 13};
\node[above=1.0pt of 14] {\fontsize{6.4}{7.4}\selectfont N: 14};
\node[above=1.0pt of 15] {\fontsize{6.4}{7.4}\selectfont N: 15};
\node[above=1.0pt of 16] {\fontsize{6.4}{7.4}\selectfont N: 16};
\node[above=1.0pt of 17] {\fontsize{6.4}{7.4}\selectfont N: 17};
\node[above=1.0pt of 18] {\fontsize{6.4}{7.4}\selectfont N: 18};
\node[above=1.0pt of 19] {\fontsize{6.4}{7.4}\selectfont N: 19};
\node[above=1.0pt of 20] {\fontsize{6.4}{7.4}\selectfont N: 20};
\node[above=1.0pt of 21] {\fontsize{6.4}{7.4}\selectfont N: 21};
\node[draw, draw=black!45, line width=0.4pt, rounded corners=3pt, inner xsep=6pt, inner ysep=8.5pt, fit=(6)(7)(8)(9), yshift=7.2pt] (group2) {};
\draw[draw=black!45, line width=0.4pt] (2.south) -- (2.south |- group2.north);
\node[draw, draw=black!45, line width=0.4pt, rounded corners=3pt, inner xsep=6pt, inner ysep=8.5pt, fit=(10)(11)(12)(13), yshift=7.2pt] (group3) {};
\draw[draw=black!45, line width=0.4pt] (3.south) -- (3.south |- group3.north);
\node[draw, draw=black!45, line width=0.4pt, rounded corners=3pt, inner xsep=6pt, inner ysep=8.5pt, fit=(14)(15)(16)(17), yshift=7.2pt] (group4) {};
\draw[draw=black!45, line width=0.4pt] (4.south) -- (4.south |- group4.north);
\node[draw, draw=black!45, line width=0.4pt, rounded corners=3pt, inner xsep=6pt, inner ysep=8.5pt, fit=(18)(19)(20)(21), yshift=7.2pt] (group5) {};
\draw[draw=black!45, line width=0.4pt] (5.south) -- (5.south |- group5.north);
\draw[-, decorate, decoration={brace, amplitude=5pt}, thick, draw=black!55] (-8.6, -1.79) -- (-8.6, 0.75) node[midway, left=6pt, font=\scriptsize, rotate=90, anchor=south] {Stage 1};
\draw[-, decorate, decoration={brace, amplitude=5pt}, thick, draw=black!55] (-8.6, -4.78) -- (-8.6, -2.6) node[midway, left=6pt, font=\scriptsize, rotate=90, anchor=south] {Stage 2};
\draw[-, decorate, decoration={brace, amplitude=5pt}, thick, draw=black!55] (-8.6, -16.82) -- (-8.6, -5.2) node[midway, left=6pt, font=\scriptsize, rotate=90, anchor=south] {Stage 3};
\end{tikzpicture}
\end{adjustbox}
\vspace{-0.03\textwidth}
\caption{Decision tree for the base model.}
\label{fig:base-model-tree}
\end{figure}

\vspace{-0.2cm}

The investment plans adapt to the realized technological advancements, differing across scenario paths in both the timing and the mix of installed technologies. The differences are clearest in solar deployment. Node~4, where solar advances quickly but lithium-ion batteries do not, installs 156~\si{\mega\watt} of solar already in the second stage, compared with 108~\si{\mega\watt} at node~2, where both technologies advance slowly; the four scenario paths descending from node~4 then expand photovoltaic capacity further in the final stage, reaching the largest cumulative solar capacity in the tree, 402~\si{\mega\watt} on path~$S_{11}$. Conversely, the slower branches invest less in solar: paths $S_1$, $S_2$, and $S_4$, the children of node~2, add no photovoltaic capacity beyond the second stage and instead rely on storage, whose build-up is concentrated in the final stage and thereby shifts capacity toward the later years. The technology mix shifts as well; for example, wind turbine capacities are larger and installed earlier along trajectories with slower photovoltaic and lithium-ion progress.

These differences are also evident at the node level. Among the children of node~2 (nodes 6, 7, 8, and 9), only one node installs electricity generation technologies, whereas all other leaf nodes do. Budget utilization varies in a similar way: the budget constraints are active in every year of nodes 2, 6, 10, 14, and 18, a pattern most pronounced along the scenario paths in which technology development proceeds slowly during the transition from the second to the third stage. Along the first scenario path ($S_1$), for instance, the budget constraints remain active throughout the final decade, whereas no technology is installed in several years at nodes 3 and 5. Even leaf nodes branching from the same parent can exhibit considerably different technology mixes: given slow lithium-ion progress from the second to the third stage, photovoltaic installation capacities are 6 to 252~\si{\mega\watt} higher when the solar advancement in the second stage is fast. Together, these adaptations illustrate the sensitivity of the investment plan to the technological advancement trajectories and demonstrate the value of the multi-stage stochastic formulation, which tailors the timing and the technology mix to the evolving costs and efficiencies.

\subsubsection{Out-of-Sample Performance Evaluation}
\label{subsubsec:oos-performance}

Table~\ref{tab:monte-carlo-results} summarizes the out-of-sample performance of the nominal and robust investment plans, obtained by solving model~\eqref{eq:stochasticModel} at investment robustness levels $\epsilon_{i} \in \{0, 0.1, 0.2\}$. Each plan is evaluated on a broader range of randomly generated operational data through the rolling-horizon Monte Carlo simulations described in Section~\ref{subsec:rolling-horizon-sim}. The operational deviations follow either a white-noise or a pink-noise process, and each plan is dispatched at robustness levels $\epsilon_{d} \in \{0, 1, 2\}$ applied to the lookahead sub-periods. In both the robust reformulation and the robust dispatch problems, the box uncertainty intervals are defined in terms of standard deviations: the deviation $\hat{v}_{d,i}$ that sets the half-width of each interval equals the standard deviation $s_d s_{c,i}$ derived in Section~\ref{subsec:operational-uncertain-data}, so each uncertain parameter is displaced from its nominal value by $\epsilon$ standard deviations. The simulations use dispatch windows of $\hat{l} \smalleq 12$ sub-periods, of which the first $\bar{l} \in \{1, 3\}$ are treated as known.

Each configuration is evaluated over $100$ replications solved in parallel. A single replication comprises 458,640 linear programs and takes approximately 550 seconds. The table reports the cost and quantity of purchased electricity and heat over the planning horizon, together with the net-zero shortfalls, expressed as the percentage of final-period electricity and heat demand met by exogenous purchases. All reported quantities are averaged across scenario paths and replications.

Increasing the investment robustness level $\epsilon_{i}$ trades a higher planned cost for more reliable out-of-sample performance. Building the plan under more conservative operational profiles raises its expected cost from \$162.27M for the nominal model ($\epsilon_{i} \smalleq 0$) to \$167.22M and \$172.81M for the robust models ($\epsilon_{i} \smalleq 0.1$ and $\epsilon_{i} \smalleq 0.2$), increases of 3.0\% and 6.5\%, respectively. In return, the out-of-sample purchasing cost falls by roughly 8 to 9\% in every configuration. The net-zero shortfalls respond in the same way: raising $\epsilon_{i}$ lowers them in nearly all configurations, the only exception being the already negligible electricity shortfall under white noise with $\bar{l} \smalleq 3$, where the heat shortfall nonetheless continues to decline. These shortfalls remain small in absolute terms throughout, so the plans rarely miss the net-zero target by a wide margin. The gains are uneven across commodities, however: heat shortfalls remain consistently larger than electricity shortfalls and are harder to drive to zero. This mirrors the smaller proportional reduction in heat purchases and identifies heat as the binding commodity for the installed technology mix.

\begin{table}[H]
\centering
\footnotesize
\setlength{\tabcolsep}{4pt}
\renewcommand{\arraystretch}{1.0}
\caption{Out-of-sample Monte Carlo simulation results for the nominal and robust investment plans under white- and pink-noise deviations.}
\label{tab:monte-carlo-results}
\resizebox{\textwidth}{!}{%
\begin{tabular}{ccc ccccr @{\hskip 10pt} ccccc}
\toprule
& & & \multicolumn{5}{c}{$\boldsymbol{\bar{l} \smalleq 1}$} & \multicolumn{5}{c}{$\boldsymbol{\bar{l} \smalleq 3}$} \\
\cmidrule(lr){4-8} \cmidrule(lr){9-13}
\multirow{2}{*}{\makecell{Investment\\$\epsilon_{i}$}} & \multirow{2}{*}{Noise} & \multirow{2}{*}{\makecell{Dispatch\\$\epsilon_{d}$}}
    & \multirow{2}{*}{\makecell{Purchasing\\cost (\$M)}} & \multicolumn{2}{c}{Purchased (MWh)} & \multicolumn{2}{c}{Shortfall (\%)}
    & \multirow{2}{*}{\makecell{Purchasing\\cost (\$M)}} & \multicolumn{2}{c}{Purchased (MWh)} & \multicolumn{2}{c}{Shortfall (\%)} \\
\cmidrule(lr){5-6} \cmidrule(lr){7-8} \cmidrule(lr){10-11} \cmidrule(lr){12-13}
 & & & & \multicolumn{1}{c}{Electricity} & \multicolumn{1}{c}{Heat} & \multicolumn{1}{c}{Electricity} & \multicolumn{1}{c}{Heat}
    & & \multicolumn{1}{c}{Electricity} & \multicolumn{1}{c}{Heat} & \multicolumn{1}{c}{Electricity} & \multicolumn{1}{c}{Heat} \\
\midrule
\multirow{6}{*}{$0$}
 & \multirow{3}{*}{White} & $0$ & 35.03 & 88,948 & 746,917 & 0.152 & 0.398 & 34.58 & 87,448 & 736,730 & 0.034 & 0.072 \\
 & & $1$ & 34.76 & 86,775 & 745,658 & 0.006 & 0.064 & 34.56 & 86,759 & 738,620 & 0.002 & 0.039 \\
 & & $2$ & 35.08 & 86,497 & 757,951 & 0.000 & 0.078 & 34.70 & 86,556 & 744,526 & 0.000 & 0.046 \\
\cmidrule(lr){2-13}
 & \multirow{3}{*}{Pink} & $0$ & 35.78 & 93,550 & 757,007 & 1.309 & 2.363 & 35.00 & 89,784 & 742,430 & 0.251 & 0.839 \\
 & & $1$ & 35.16 & 88,868 & 751,813 & 0.202 & 0.877 & 34.88 & 88,420 & 743,380 & 0.049 & 0.613 \\
 & & $2$ & 35.38 & 88,083 & 762,360 & 0.037 & 0.637 & 35.01 & 88,113 & 749,050 & 0.027 & 0.592 \\
\midrule
\multirow{6}{*}{$0.1$}
 & \multirow{3}{*}{White} & $0$ & 33.08 & 78,120 & 728,475 & 0.131 & 0.281 & 32.55 & 76,166 & 717,145 & 0.037 & 0.049 \\
 & & $1$ & 32.76 & 75,525 & 727,145 & 0.004 & 0.033 & 32.51 & 75,447 & 718,704 & 0.002 & 0.017 \\
 & & $2$ & 33.10 & 75,259 & 740,205 & 0.000 & 0.035 & 32.65 & 75,293 & 724,433 & 0.000 & 0.019 \\
\cmidrule(lr){2-13}
 & \multirow{3}{*}{Pink} & $0$ & 33.87 & 83,172 & 737,764 & 1.088 & 1.967 & 33.01 & 78,928 & 722,733 & 0.242 & 0.648 \\
 & & $1$ & 33.19 & 77,966 & 732,764 & 0.153 & 0.641 & 32.87 & 77,435 & 723,175 & 0.039 & 0.426 \\
 & & $2$ & 33.43 & 77,154 & 744,097 & 0.025 & 0.433 & 32.99 & 77,157 & 728,608 & 0.019 & 0.405 \\
\midrule
\multirow{6}{*}{$0.2$}
 & \multirow{3}{*}{White} & $0$ & 31.93 & 73,040 & 709,055 & 0.121 & 0.244 & 31.43 & 71,443 & 697,862 & 0.038 & 0.047 \\
 & & $1$ & 31.65 & 70,664 & 708,453 & 0.004 & 0.022 & 31.39 & 70,667 & 699,557 & 0.002 & 0.010 \\
 & & $2$ & 32.04 & 70,314 & 723,484 & 0.000 & 0.017 & 31.55 & 70,395 & 706,128 & 0.000 & 0.010 \\
\cmidrule(lr){2-13}
 & \multirow{3}{*}{Pink} & $0$ & 32.73 & 77,890 & 719,404 & 1.034 & 1.767 & 31.91 & 74,125 & 704,234 & 0.237 & 0.563 \\
 & & $1$ & 32.10 & 73,037 & 714,874 & 0.143 & 0.544 & 31.77 & 72,651 & 704,730 & 0.034 & 0.349 \\
 & & $2$ & 32.39 & 72,200 & 728,013 & 0.021 & 0.348 & 31.90 & 72,256 & 710,971 & 0.015 & 0.325 \\
\bottomrule
\end{tabular}%
}
\end{table}

\vspace{-0.2cm}

Beyond the choice of investment plan, the temporal structure of the operational deviations has a pronounced effect on out-of-sample performance. Under white noise, the deviations are temporally independent and largely offset one another across sub-periods whereas under pink noise, unfavorable realizations, such as below-nominal renewable generation and above-nominal demand, can persist over consecutive sub-periods, leaving the installed technology mix unable to fully meet demand. Pink noise therefore raises purchased electricity and heat in every configuration. 
Two levers mitigate this effect: Extending the known portion of the window from $\bar{l} \smalleq 1$ to $\bar{l} \smalleq 3$ lets each dispatch problem anticipate its upcoming sub-periods. When no dispatch robustness is applied, this more than halves the shortfalls in every configuration and also lowers the purchased quantities, especially under pink noise. Dispatch-side robustness provides a complementary hedge that leaves the investment plan unchanged, although its effect on both the net-zero shortfalls and the purchasing cost is non-monotone. The first increment ($\epsilon_{d} \smalleq 1$) reduces both the shortfalls and the purchasing cost in every configuration, whereas the second ($\epsilon_{d} \smalleq 2$) dispatches so conservatively that it raises heat purchases enough to increase the overall purchasing cost. This reversal arises because the second increment doubles the width of the dispatch problem's uncertainty intervals, bracing each lookahead sub-period against more pessimistic combination of low renewable generation and a low heat-pump coefficient of performance. Anticipating these harsh conditions, the dispatch draws down its stored electricity in the current sub-periods rather than holding it in reserve, leaving less available for the periods that follow. When these conservative expectations go unrealized, the depleted storage can no longer drive the heat pump, so the unmet heat demand is covered by exogenous purchases, which raises the overall purchasing cost.

\subsubsection{Sensitivity Analyses}
\label{subsubsec:sensitivity-analyses}

We analyze the effect of three key problem parameters on the transition plan: the final-period emission target, the investment budget, and the energy demand.

\textit{The Effect of Emission Target:}
The first analysis concerns the sensitivity of the expected cost of the transition plan to the final-period emission target. The base-case instance is re-solved with $\gamma_{r} \in \{0.01, 0.02, 0.03, 0.04, 0.05\}$ for all $r \in \mathcal{R}$, allowing the corresponding fraction of final-period demand in each energy category to be met by exogenous purchases. The validity of the inequalities~\eqref{eq:validineq10} relies on the elimination of the exogenous procurement decision variables, which is possible only when $\gamma_{r} \smalleq 0$. These instances are therefore solved using \textit{Algorithm~\ref{alg:benders-decomposition} without VIs}.

Table~\ref{tab:emission-target-sensitivity} reports the expected cost across the range of exogenous procurement allowances. The results reveal a clear trade-off between the stringency of the emission target and the cost of the transition plan: permitting exogenous purchases for just 1\% of final-period demand lowers the expected cost from the base-case \$162.27M to \$125.67M, while each additional percentage point yields a progressively smaller reduction, with the cost reaching \$108.80M at the 5\% allowance. The 
large saving from the first percentage point indicates that insisting on an exact net-zero target 
carries a substantial cost premium. The allowances also reshape the technology mix, with the initial allowances substantially reducing the installed storage capacity and further increments primarily lowering the installed capacities of generation technologies and heat pumps. This occurs since purchases meet a larger share of demand in sub-periods with unfavorable renewable performance, lowering the capacity that would otherwise be required to serve it.

\begin{table}[H]
\centering
\footnotesize
\renewcommand{\arraystretch}{1.05}
\caption{Expected costs under different final-period exogenous procurement allowances.}
\label{tab:emission-target-sensitivity}
\begin{tabular}{lcccccc}
\toprule
$\gamma_{r}$, $r \in \mathcal{R}$ & $0$ & $0.01$ & $0.02$ & $0.03$ & $0.04$ & $0.05$ \\
\midrule
Expected cost (\$M) & 162.27 & 125.67 & 118.61 & 114.51 & 111.36 & 108.80 \\
\bottomrule
\end{tabular}
\end{table}


\vspace{-0.2cm}
\textit{The Effect of Investment Budget:}
The annual budget, fixed at \$20M in every period in the base case, is varied next. Raising it to \$25M reduces the expected cost from the base case to \$153.38M, and a further increase to \$30M lowers it to \$149.07M. The additional reduction becomes smaller as the budget grows. These diminishing marginal savings reflect the budget acting as a binding constraint: as noted for the base-case decision tree (Figure~\ref{fig:base-model-tree}), the budget caps are active in many years, particularly along the scenario paths with slow technology development. Relaxing the annual budget loosens these caps, so that capacity can be installed closer to when it is needed and at lower cost, rather than earlier to remain within the annual limit.

\textit{The Effect of Energy Demand:}
Although no clear driver of demand growth is anticipated for the METU campus, the robustness of the transition plan to such growth is nonetheless examined by increasing the nominal electricity and heat demand at compounding annual rates of 0.5\% and 1\% over the planning horizon. The expected cost rises from the base case to \$180.32M under 0.5\% growth and to \$202.89M under 1\% growth, with the increase accelerating at the higher growth rate. Because the demand increments compound over the horizon and exogenous purchases are prohibited in the final period under the net-zero target, the additional final-period demand must be met entirely by the installed capacity. The plan therefore installs considerably more renewable generation and storage capacity. Since the budget binds during the final decade along some scenario paths, part of this additional capacity is installed in the early years, when technologies are more expensive, which adds further to the expected cost.

\section{Conclusion}
\label{sec:conclusion-main-findings}

This paper studies clean electricity--heat transition planning for a campus-scale integrated energy system, where long-term investment decisions must adapt to uncertain technological progress while remaining compatible with high-resolution operations. We represent this setting with a multi-stage stochastic mixed-integer program and develop a robust reformulation that captures operational uncertainty through box uncertainty sets. To solve the resulting large-scale instances, we propose a Benders decomposition algorithm with partial non-anticipativity relaxation: operational non-anticipativity is relaxed during the decomposition and restored only at termination. The algorithm is further strengthened with 
valid inequalities and a two-phase cut-addition strategy.

The METU campus case study shows that the proposed algorithm substantially improves computational efficiency. It solves every instance, across all temporal resolutions, to a 1\% optimality gap within the time limit, whereas both the extensive form and classical Benders decomposition solve only the coarsest 24-hour instance. On that instance, the proposed algorithm is roughly an order of magnitude faster than the extensive form and more than thirty times faster than classical Benders. The resulting net-zero transition plan has an expected cost of \$162.27M and relies primarily on photovoltaic panels, lithium-ion batteries, and heat pumps.

The rolling-horizon Monte Carlo simulations and sensitivity analyses provide several planning insights. Increasing the investment robustness level raises the planned transition cost, but it reduces out-of-sample purchasing costs and net-zero shortfalls. Temporally correlated operational deviations are more challenging than independent deviations, especially for heat supply, while dispatch-side robustness can mitigate these effects when its conservatism is chosen carefully. The final-period emission target is the dominant cost driver: enforcing an exact net-zero target, rather than allowing even a marginal final-period shortfall, carries a substantial cost premium.

Beyond the campus application, the methodology applies to multi-stage stochastic mixed-integer programs with a similar master--subproblem structure. In such problems, the master problem retains variables whose non-anticipativity is difficult to restore, while the decomposition assigns the remaining variables to scenario-wise subproblems and corrects their non-anticipativity after termination. The correction step must recover a feasible non-anticipative solution while limiting any change in the objective value. In the present model, the correction LP satisfies this requirement and restores non-anticipativity without increasing the objective value across the computational instances. A natural next step is to scale the framework to larger multi-energy networks, supporting transition planning beyond the campus setting.

\subsection*{Acknowledgments}
This work was supported by the  Scientific and Technological Research Council of Turkey [grant number 222M243].

{\begingroup
\setstretch{1}
\footnotesize
\bibliography{thermalbib}
\endgroup
}

\appendix

\section{Data}
\label{sec:appendixA}

Table~\ref{tab:VersionsTable} presents the annual energy yields of the candidate renewable technologies, computed from hourly generation profiles simulated with the National Renewable Energy Laboratory’s System Advisor Model (SAM)~\citep{nrelsam}. For each technology, SAM uses the geographic location of the METU campus and the technical specifications listed in Table~\ref{tab:VersionsTable} to simulate hourly energy output over a representative meteorological year. The simulated generation profiles, demand series, and detailed input data used to construct the model are publicly available in the source code repository~\citep{stochastic-campus-energy-transition}.

\begin{table}[H]
\centering
\caption{Candidate renewable technologies (costs are in 2024~\$).}
\label{tab:VersionsTable}
\begin{adjustbox}{max width=\textwidth}
\renewcommand{\arraystretch}{1.3}
\begin{tabular}{|c|c|c|c|c|c|c|c|c|c|c|}
\hline
\makecell[c]{\textbf{Technology}} &
\makecell[c]{\textbf{Capacity}} &
\makecell[c]{\textbf{Installation}\\\textbf{Cost}} &
\makecell[c]{\textbf{Annual}\\\textbf{O\&M}\\\textbf{Cost}$^{\dagger}$} &
\makecell[c]{\textbf{Footprint}} &
\makecell[c]{\textbf{Annual}\\\textbf{Energy}\\\textbf{Yield}} &
\makecell[c]{\textbf{Economic}\\\textbf{Lifetime}} &
\makecell[c]{\textbf{Annual}\\\textbf{Degradation}\\\textbf{Rate}} &
\makecell[c]{\textbf{Round-trip}\\\textbf{Efficiency}} &
\makecell[c]{\textbf{Self-discharge}\\\textbf{Rate}} &
\makecell[c]{\textbf{Reference}} \\
\hline
Solar PV$_{1}$ & \SI{6}{\mega\watt}  & \$3{,}190{,}170 & \multirow{2}{*}{\makecell[c]{8.35~\$/\si{\kilo\watt} --\\0.918}} & 28{,}370 \si{\meter\squared} & \SI{9.63}{\giga\watt\hour} & \multirow{2}{*}{25 years} & \multirow{2}{*}{0.5\%} & \multirow{2}{*}{-} & \multirow{2}{*}{-} & \multirow{2}{*}{\citet{csener2025dynamic}} \\
\cline{1-3} \cline{5-6}
Solar PV$_{2}$ & \SI{12}{\mega\watt} & \$5{,}832{,}190 & & 56{,}739 \si{\meter\squared}  & \SI{19.26}{\giga\watt\hour} & & & & & \\
\hline
\makecell[c]{Wind\\Turbine$_{1}$} & \SI{6}{\mega\watt} & \$5{,}430{,}435 & \makecell[c]{13.74~\$/\si{\kilo\watt} --\\0.898} & 86{,}409 \si{\meter\squared} & \SI{5.47}{\giga\watt\hour} & 25 years & 1.6\% & - & - & \citet{csener2025dynamic} \\
\hline
\makecell[c]{Parabolic\\Trough$_{1}$} & \SI{4.02}{\mega\watt} & \$2{,}812{,}480 & 13.99~\$/\si{\kilo\watt} & 5{,}280 \si{\meter\squared} & \SI{4.15}{\giga\watt\hour} & 25 years & 0.2\% & - & - & \citet{nrelsam} \\
\hline
Heat Pump$_{1}$ & \SI{246}{\kilo\watt\hour} & \$150{,}563 & 9.18~\$/\si{\kilo\watt\hour} & \SI[locale=US,group-digits=false]{12}{\meter\squared} & - & 20 years & 2.0\% & - & - & \citet{Mavromatidis2021} \\
\hline
\makecell[c]{Lithium-ion\\Battery$_{1}$}  & \SI{1}{\mega\watt\hour} & \$305{,}556 & 5.50~\$/\si{\kilo\watt\hour} & \SI[locale=US,group-digits=false]{30.0}{\meter\squared} & - & 20 years & 1.0\% & 90.0\% & 0.06\%/\si{\hour} & \citet{csener2025dynamic} \\
\hline
\makecell[c]{Hot Water\\Storage$_{1}$} & \SI{1}{\mega\watt\hour} & \$15{,}000 & 0.30~\$/\si{\kilo\watt\hour} & \SI[locale=US,group-digits=false]{8.7}{\meter\squared} & - & 30 years & - & 90.0\% & 0.50\%/\si{\hour} & \citet{Mavromatidis2021} \\
\hline
\end{tabular}
\end{adjustbox}
\vspace{0.15cm}
\begin{minipage}{\linewidth}
{\footnotesize\setlength{\baselineskip}{0.8\baselineskip}\noindent $^{\dagger}$O\&M cost shown as: initial annual cost (\$/kW) -- annual reduction multiplier \par}
\end{minipage}
\end{table}

\begin{table}[H]
\centering
\caption{Five-year technology advancement multipliers.}
\label{tab:cluster_summary}
\footnotesize
\begin{adjustbox}{max width=\textwidth}
\begin{tabular}{l c c c c}
\toprule
\textbf{Technology} & \textbf{Advancement Probability} & \textbf{Cost Multiplier} & \textbf{Efficiency Multiplier} & \textbf{Reference} \\
\midrule
\multirow{2}{*}{Solar photovoltaic}  & Fast (0.667) & 0.551 & 1.133 & \multirow{2}{*}{\citet{csener2025dynamic}} \\
& Slow (0.333) & 0.856 & 1.056 & \\
\addlinespace
\multirow{2}{*}{Lithium-ion battery} & Fast (0.464) & 0.348 & 1.000 & \multirow{2}{*}{\citet{csener2025dynamic}} \\
& Slow (0.536) & 0.634 & 1.000 & \\
\addlinespace
Wind turbine & Medium (1.000) & 0.815 & 1.000 & \citet{csener2025dynamic} \\
Parabolic trough & Medium (1.000) & 0.683 & 1.000 & \citet{irena2025rpgc} \\
Heat pump & Medium (1.000) & 0.874 & $1.156^\dagger$ & \citet{Mavromatidis2021} \\
Thermal storage & Medium (1.000) & 1.000 & 1.000 & \citet{Mavromatidis2021} \\
\bottomrule
\end{tabular}
\end{adjustbox}
\vspace{0.15cm}
\begin{minipage}{\linewidth}
{\footnotesize\setlength{\baselineskip}{0.8\baselineskip}\noindent $^{\dagger}$The advancement multipliers are applied multiplicatively across successive five-year periods, whereas the heat pump efficiency multiplier is applied additively. \par}
\end{minipage}
\end{table}

\end{document}